\documentclass[10pt]{amsart}                          
\usepackage{epsfig}            
\usepackage{amssymb,amscd,bbm}     
\usepackage{amsthm,amsgen,amsmath,empheq}
\usepackage[dvips]{xypic}       %  Needed for graphs to draw right
\usepackage{mathdots}
\usepackage{stmaryrd}

\xyoption{curve}                %  more power to xypic
\xyoption{rotate}               %    ditto

\newcommand{\graphpp}[3]{ \ensuremath{%  3-graph 
 \begin{xy}                           %    _ b 
  (0,-2)*+U{\scriptstyle #1}="a",    %    /| \
  (3,3)*+U{\scriptstyle #2}="b",     %   /   _\|
  (6,-2)*+U{\scriptstyle #3}="c",    %  a      c
  "a";"b"**\dir{-}?>*\dir{>},         %
  "b";"c"**\dir{-}?>*\dir{>}          %
 \end{xy}                             %
} }

\newcommand{\linep}[2]{ \ensuremath{  %
 \begin{xy}                           %  2-graph 
  (0,-2)*+U{\scriptstyle #1}="a",     %   _ b
  (6,2)*+U{\scriptstyle #2}="b",      %   /|
  "a";"b"**\dir{-}?>*\dir{>},         %  a
 \end{xy}                             %
} }

\theoremstyle{plain}                          
\newtheorem{theorem}{Theorem}[section]                          
\newtheorem{proposition}[theorem]{Proposition}                          
\newtheorem{lemma}[theorem]{Lemma}                          
\newtheorem{corollary}[theorem]{Corollary}

\theoremstyle{definition}                          
\newtheorem{definition}[theorem]{Definition}  

\theoremstyle{remark}                          
\newtheorem{remark}[theorem]{Remark}

\newtheorem{example}[theorem]{Example}

\begin{document}

\title{The Configuration Basis of a Lie Algebra and its Dual}

\author[B. Walter]{Ben Walter} 
\address{
Department of Mathematics \\ Middle East Technical University, Northern Cyprus Campus \\
Kalkanli, Guzelyurt, KKTC, Mersin 10 Turkey
}
\email{benjamin@metu.edu.tr}


\date{\today}

\begin{abstract}
We use the Lie coalgebra and configuration pairing framework presented previously in 
\cite{SiWa07} to derive a new, left-normed monomial basis for free Lie algebras
(built from associative Lyndon-Shirshov words), 
as well as a dual monomial basis for Lie coalgebras.   
Our focus is on computational dexterity gained by using the configuration framework and basis.
We include several explicit examples using the dual coalgebra basis and configuration pairing
to perform Lie algebra computations.  As a corollary of our work, we get a new multiplicative 
basis for the shuffle algebra.
\end{abstract}

\maketitle

\section{Introduction}

Let $V$ be a finite dimensional vector space with ordered basis 
$v_1<\cdots<v_d$.  The (associative) Lyndon-Shirshov words on the alphabet $v_1,\dots,v_d$
are finite words which are 
lexicographically minimal among their cyclic permutations.
Various different methods 
(see e.g. \cite{CHL58}, \cite{Chib06}) %, \cite{Stoh08})
are present in the literature, constructing different
bases for the free Lie algebra on $V$ starting from Lyndon-Shirshov 
words.  We will derive a new basis for $\mathrm{L}V$ from 
Lyndon-Shirshov words, as well
as a dual basis of monomials in the graph presentation of the 
cofree conilpotent Lie coalgebra
$\mathrm{E}V^*$ on $V^*$ using the configuration
pairing and graph coalgebras framework of \cite{SiWa07}.

Since our basis is natural from the point of view of the configuration
pairing, we call it the configuration basis of a Lie algebra.
We emphasize that the surprising property of the configuration basis 
is that it has a dual Lie coalgebra basis of graph monomials.
While some previous work has considered duality (e.g. \cite{MeRe89}) 
in the formation of bases for Lie algebras, previously duality was formal,
rather than using explicit presentations, and work occured in the universal
enveloping algebra of Lie algebras, rather than on the Lie algebras themselves.

Note that even though our starting point is the set of Lyndon words, 
the configuration basis is not a Hall basis.  However, it still satisfies some
of the same nice properties as Hall bases.  For example, 
writing Lie bracket expressions in terms of the configuration basis yields 
integer coefficients (Corollary~\ref{C:integer coefficients}) and 
the Lie polynomial of the configuration basis can be written as shuffles of 
higher ordered elements (Theorem~\ref{P:lower triangular}).
The latter property implies that the configuration basis may be used for
Gr\"obner basis calculations.

\medskip

The paper is structured as follows.  
In Section~\ref{S:background} we recall the  
Lie coalgebras framework introduced in \cite{Sinh06.2} and \cite{SiWa07}.  
We organize ideas slightly differently than \cite{SiWa07} and give extra 
computational examples to clarify the framework and its use in this setting.
The central idea needed is the configuration pairing of Lie algebras and 
coalgebras.  This will be described by writing Lie algebra elements as
trees and Lie coalgebra elements as graphs.

In Section~\ref{S:B} we introduce a new grading on the set of Lyndon-Shirshov
words $\mathcal{B}$ \cite{Chib06}, \cite{Reut93}.
A simple word is $x\dots xy_1\dots y_\ell$ with $x\neq y_i$.  
Words are graded by classifying them as simple words, 
simple words of simple words,
simple words of those, etc.

In Section~\ref{S:main} we recursively define bracketing and graph maps
$\mathcal{L}$ and $\mathcal{G}$ making Lie algebra and coalgebra elements
from Lyndon-Shirshov words.  Our main theorem shows that $\mathcal{L}\mathcal{B}$ and
$\mathcal{G}\mathcal{B}$ are dual monomial bases for the free Lie algebra 
and cofree conilpotent
Lie coalgebra via the configuration pairing.  The set $\mathcal{L}\mathcal{B}$ 
is the configuration basis of the free Lie algebra $\mathrm{L}V$.

In Sections~\ref{S:examples} and \ref{S:Lyndon} we give examples using the 
configuration pairing to write Lie expressions in terms of the configuration
basis.  We also recall the classical Harrison model of conilpotent Lie coalgebras
as words modulo shuffles and note that in this model the Lyndon words themselves
form a basis.  In examples, we show how the configuration pairing, which is not 
part of the classical picture, can be used to perform Lie algebra calculations
using the Lyndon word basis for the Harrison model of $\mathrm{E}V^*$.

In Sections~\ref{S:comparison} and \ref{S:future} we compare the configuration 
basis to some other bases present in the literature, 
connecting with classical Lie models, and
outline some possible avenues of future interest.  In Section~\ref{S:new basis} 
we also construct an 
alternate set of words $\hat{\mathcal{B}}$ which can be used as a replacement 
for the Lyndon-Shirshov words.  The set $\hat{\mathcal{B}}$ is also a new multiplicative
basis for the shuffle algebra.

\section{Graphs, Lie Coalgebras, and the Configuration Pairing}\label{S:background}

We begin with a summary of pertinent results from \cite{SiWa07} and \cite{Sinh06.2} 
slightly rephrased and specialized for the convenience of our setting.  
Throughout this section we will say tree and graph to mean 
rooted, binary tree embedded in the upper-half plane
and oriented, connected, acyclic graph. 
Trees will be denoted $T$ and graphs, $G$.
Labeled trees and graphs,  $\tau=(T,l_T)$ and $\gamma=(G,l_G)$ are trees and graphs
with maps  $l_T:\mathrm{Leaves}(T)\to S_T$ and $l_G:\mathrm{Vertices}(G)\to S_G$ 
where $S_T,S_G$ are labeling sets.  

Our shorthand for writing labeled trees and graphs will be to write corresponding labels
in place of their leaves and vertices.
We do not require $l_T$ or $l_G$ to be injective or surjective.

Let $V$ be a finite dimensional vector space with dual $V^*$.
Write $\mathrm{Tr}(V)$ and $\mathrm{Gr}(V^*)$ for the vector spaces 
generated by trees and graphs 
with labels from $V$ and $V^*$, modulo multilinearity in the labels.  
The vector space $\mathrm{Tr}(V)$ has a standard product defined on monomials as 
$\tau_1\otimes \tau_2 \longmapsto [\tau_1,\tau_2] = 
\begin{aligned}\begin{xy}
	(0,1.5); (1.5,0)**\dir{-};          
  (3,1.5)**\dir{-},               
  (1.5,-1.5); (1.5,0)**\dir{-},       
  (-.4,2.7)*{\scriptstyle \tau_1},   
  (3.8,2.7)*{\scriptstyle \tau_2},   
\end{xy}\end{aligned}
$. 
The vector space $\mathrm{Gr}(V^*)$ has an anti-commutative coproduct 
defined on monomials as 
$\gamma \longmapsto\ ]\gamma[\ 
= \sum_e \gamma^{\hat e}_1 \otimes \gamma^{\hat e}_2 - \gamma^{\hat e}_2 \otimes \gamma^{\hat e}_1$ 
where $\sum_e$ is a sum over all edges of $G$ and $\gamma^{\hat e}_1$, $\gamma^{\hat e}_2$ are the graphs
obtained from $\gamma$ by removing the edge $e$ which points to the 
subgraph $\gamma^{\hat e}_2$.

It is a standard fact that $\mathrm{L}V$,
the free Lie algebra on $V$,  is isomorphic as algebras to $\mathrm{Tr}(V)$
modulo the locally defined anti-symmetry and Jacobi relations: 
\begin{align*}
\text{(anti-symmetry)} \qquad 
 & \qquad
\begin{xy}
   (0,1.5); (1.5,0)**\dir{-};          
   (3,1.5)**\dir{-},               
   (1.5,-1.5); (1.5,0)**\dir{-},       
   (-.4,2.7)*{\scriptstyle T_1},   
   (3.8,2.7)*{\scriptstyle T_2},   
   (1.5,-2.7)*{\scriptstyle R}
\end{xy}\ =\ - \begin{xy}
   (0,1.5); (1.5,0)**\dir{-};          
   (3,1.5)**\dir{-},               
   (1.5,-1.5); (1.5,0)**\dir{-},       
   (-.4,2.7)*{\scriptstyle T_2},   
   (3.8,2.7)*{\scriptstyle T_1},    
   (1.5,-2.7)*{\scriptstyle R}
\end{xy} \\
\text{(Jacobi)} \qquad 
 & \qquad 
 \begin{xy}   
   (1.5,1.5); (3,3)**\dir{-}, 
   (0,3); (3,0)**\dir{-};
   (6,3)**\dir{-},   
   (3,0); (3,-1.5)**\dir{-}, 
   (-.4,4.2)*{\scriptstyle T_1}, 
   (3.2,4.2)*{\scriptstyle T_2},
   (6.8,4.2)*{\scriptstyle T_3}, 
   (3,-2.7)*{\scriptstyle R}
 \end{xy}  \ + \ \begin{xy}   
   (1.5,1.5); (3,3)**\dir{-}, 
   (0,3); (3,0)**\dir{-};
   (6,3)**\dir{-},   
   (3,0); (3,-1.5)**\dir{-}, 
   (-.4,4.2)*{\scriptstyle T_2}, 
   (3.2,4.2)*{\scriptstyle T_3},
   (6.8,4.2)*{\scriptstyle T_1}, 
   (3,-2.7)*{\scriptstyle R}
 \end{xy} \ + \ \begin{xy}   
   (1.5,1.5); (3,3)**\dir{-}, 
   (0,3); (3,0)**\dir{-};
   (6,3)**\dir{-},   
   (3,0); (3,-1.5)**\dir{-}, 
   (-.4,4.2)*{\scriptstyle T_3}, 
   (3.2,4.2)*{\scriptstyle T_1},
   (6.8,4.2)*{\scriptstyle T_2}, 
   (3,-2.7)*{\scriptstyle R}
 \end{xy} 
 \ =\  0,
\end{align*}
where $R$, $T_1$, $T_2$, $T_3$ stand for arbitrary (possibly trivial subtrees) which
are not modified in these operations.
($\mathrm{Tr}(V)$ itself is isomorphic to the free nonassociative binary algebra on $V$, aka the free magma on $V$.)
From \cite{SiWa07} the cofree conilpotent Lie coalgebra on $V^*$, written $\mathrm{E}V^*$, 
is isomorphic as coalgebras to $\mathrm{Gr}(V^*)$ modulo the locally defined 
arrow-reversing and Arnold relations:
\begin{align*}
\text{(arrow-reversing)}\qquad & \qquad
\begin{xy}                           
  (0,-2)*+UR{\scriptstyle a}="a",    
  (3,3)*+UR{\scriptstyle b}="b",     
  "a";"b"**\dir{-}?>*\dir{>},         
  (1.5,-5),{\ar@{. }@(l,l)(1.5,6)},
  ?!{"a";"a"+/va(210)/}="a1",
  ?!{"a";"a"+/va(240)/}="a2",
  ?!{"a";"a"+/va(270)/}="a3",
  "a";"a1"**\dir{-},  "a";"a2"**\dir{-},  "a";"a3"**\dir{-},
  (1.5,6),{\ar@{. }@(r,r)(1.5,-5)},
  ?!{"b";"b"+/va(90)/}="b1",
  ?!{"b";"b"+/va(30)/}="b2",
  ?!{"b";"b"+/va(60)/}="b3",
  "b";"b1"**\dir{-},  "b";"b2"**\dir{-},  "b";"b3"**\dir{-},
\end{xy}\ =\ \ -  
\begin{xy}                           
  (0,-2)*+UR{\scriptstyle a}="a",    
  (3,3)*+UR{\scriptstyle b}="b",     
  "a";"b"**\dir{-}?<*\dir{<},         
  (1.5,-5),{\ar@{. }@(l,l)(1.5,6)},
  ?!{"a";"a"+/va(210)/}="a1",
  ?!{"a";"a"+/va(240)/}="a2",
  ?!{"a";"a"+/va(270)/}="a3",
  "a";"a1"**\dir{-},  "a";"a2"**\dir{-},  "a";"a3"**\dir{-},
  (1.5,6),{\ar@{. }@(r,r)(1.5,-5)},
  ?!{"b";"b"+/va(90)/}="b1",
  ?!{"b";"b"+/va(30)/}="b2",
  ?!{"b";"b"+/va(60)/}="b3",
  "b";"b1"**\dir{-},  "b";"b2"**\dir{-},  "b";"b3"**\dir{-},
\end{xy} \\
\text{(Arnold)}\qquad & \qquad
\begin{xy}                           
  (0,-2)*+UR{\scriptstyle a}="a",    
  (3,3)*+UR{\scriptstyle b}="b",   
  (6,-2)*+UR{\scriptstyle c}="c",   
  "a";"b"**\dir{-}?>*\dir{>},         
  "b";"c"**\dir{-}?>*\dir{>},         
  (3,-5),{\ar@{. }@(l,l)(3,6)},
  ?!{"a";"a"+/va(210)/}="a1",
  ?!{"a";"a"+/va(240)/}="a2",
  ?!{"a";"a"+/va(270)/}="a3",
  ?!{"b";"b"+/va(120)/}="b1",
  "a";"a1"**\dir{-},  "a";"a2"**\dir{-},  "a";"a3"**\dir{-},
  "b";"b1"**\dir{-}, "b";(3,6)**\dir{-},
  (3,-5),{\ar@{. }@(r,r)(3,6)},
  ?!{"c";"c"+/va(-90)/}="c1",
  ?!{"c";"c"+/va(-60)/}="c2",
  ?!{"c";"c"+/va(-30)/}="c3",
  ?!{"b";"b"+/va(60)/}="b3",
  "c";"c1"**\dir{-},  "c";"c2"**\dir{-},  "c";"c3"**\dir{-},
  "b";"b3"**\dir{-}, 
\end{xy}\ + \                             
\begin{xy}                           
  (0,-2)*+UR{\scriptstyle a}="a",    
  (3,3)*+UR{\scriptstyle b}="b",   
  (6,-2)*+UR{\scriptstyle c}="c",    
  "b";"c"**\dir{-}?>*\dir{>},         
  "c";"a"**\dir{-}?>*\dir{>},          
  (3,-5),{\ar@{. }@(l,l)(3,6)},
  ?!{"a";"a"+/va(210)/}="a1",
  ?!{"a";"a"+/va(240)/}="a2",
  ?!{"a";"a"+/va(270)/}="a3",
  ?!{"b";"b"+/va(120)/}="b1",
  "a";"a1"**\dir{-},  "a";"a2"**\dir{-},  "a";"a3"**\dir{-},
  "b";"b1"**\dir{-}, "b";(3,6)**\dir{-},
  (3,-5),{\ar@{. }@(r,r)(3,6)},
  ?!{"c";"c"+/va(-90)/}="c1",
  ?!{"c";"c"+/va(-60)/}="c2",
  ?!{"c";"c"+/va(-30)/}="c3",
  ?!{"b";"b"+/va(60)/}="b3",
  "c";"c1"**\dir{-},  "c";"c2"**\dir{-},  "c";"c3"**\dir{-},
  "b";"b3"**\dir{-}, 
\end{xy}\ + \                              
\begin{xy}                           
  (0,-2)*+UR{\scriptstyle a}="a",    
  (3,3)*+UR{\scriptstyle b}="b",   
  (6,-2)*+UR{\scriptstyle c}="c",    
  "a";"b"**\dir{-}?>*\dir{>},         
  "c";"a"**\dir{-}?>*\dir{>},          
  (3,-5),{\ar@{. }@(l,l)(3,6)},
  ?!{"a";"a"+/va(210)/}="a1",
  ?!{"a";"a"+/va(240)/}="a2",
  ?!{"a";"a"+/va(270)/}="a3",
  ?!{"b";"b"+/va(120)/}="b1",
  "a";"a1"**\dir{-},  "a";"a2"**\dir{-},  "a";"a3"**\dir{-},
  "b";"b1"**\dir{-}, "b";(3,6)**\dir{-},
  (3,-5),{\ar@{. }@(r,r)(3,6)},
  ?!{"c";"c"+/va(-90)/}="c1",
  ?!{"c";"c"+/va(-60)/}="c2",
  ?!{"c";"c"+/va(-30)/}="c3",
  ?!{"b";"b"+/va(60)/}="b3",
  "c";"c1"**\dir{-},  "c";"c2"**\dir{-},  "c";"c3"**\dir{-},
  "b";"b3"**\dir{-}, 
\end{xy}\ =\ 0,                             
\end{align*}
where ${a}$, ${b}$, and ${c}$ stand for 
vertices in a graph which is fixed outside of the indicated area. 
($\mathrm{Gr}(V^*)$ itself is closely related to the cofree preLie coalgebra on $V^*$.)

\begin{remark}
	In \cite[Prop.~3.2]{SiWa07} work is restricted to graded, 1-reduced vector
	spaces.  This requirement is needed to have $\mathrm{E}V$ 
	be the cofree Lie coalgebra.  Removing graded, 1-reduced results in
	$\mathrm{E}V$ being the cofree, conilpotent Lie coalgebra.  This follows from
	\cite[Prop.~3.14]{SiWa07}, which is independent of \cite[Prop.~3.2]{SiWa07}.  The 
	difference between cofree Lie coalgebras and cofree, conilpotent Lie coalgebras
	is the presence of infinite graphs, which are not needed in the present
	application.
\end{remark}

The core of the $\mathrm{E}V \cong \mathrm{Gr}(V)/\!\sim$ proof in \cite{SiWa07}, 
and backbone of the current paper, is the configuration
pairing of graphs and trees, introduced in \cite{Sinh06.2} and extended to
$\mathrm{Gr}(V^*)$ and $\mathrm{Tr}(V)$ in \cite{SiWa07}.
Given an isomorphism $\sigma:\mathrm{Vertices}(G) \to \mathrm{Leaves}(T)$,
define $\beta_\sigma:\mathrm{Edges}(G) \to \{\text{internal vertices of $T$}\}$
by sending the edge $\linep{a}{b}$ to the internal vertex closest to the root 
of $T$ on the 
path from leaf $\sigma(a)$ to leaf $\sigma(b)$.
For unlabeled graphs and trees, the $\sigma$-configuration pairing of $G$ and $T$
is 
$$\langle G,\, T\rangle_\sigma \ = \ 
\begin{cases} \displaystyle
	\prod_{e\in\mathrm{E}(G)} \mathrm{sgn}\bigl(\beta_\sigma(e)\bigr) & \text{if $\beta_\sigma$ is surjective,}\\
	\ \ \ 0 & \text{otherwise}
\end{cases}
$$
where $\prod_e$ is a product over all edges of $G$, and 
$\mathrm{sgn}\bigl(\beta_\sigma(\linep{a}{b})\bigr) = \pm 1$ depending on whether 
leaf $\sigma(a)$ is left or right of leaf $\sigma(b)$ in the planar embedding of $T$.

\begin{example}
Following is the map $\beta_\sigma$ for two different isomorphisms $\sigma$ of the
vertices and leaves of a fixed graph and tree.  The different isomorphisms are indicated
by the numbering of the vertices and leaves.
$$
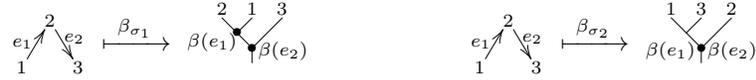
\begin{aligned}\begin{xy} 
  (0,-2.5)*+UR{\scriptstyle 1}="a", 
  (3.75,3.75)*+UR{\scriptstyle 2}="b", 
  (7.5,-2.5)*+UR{\scriptstyle 3}="c", 
  "a";"b"**\dir{-}?>*\dir{>} ?(.4)*!RD{\scriptstyle e_1}, 
  "b";"c"**\dir{-}?>*\dir{>} ?(.5)*!LD{\scriptstyle e_2} 
 \end{xy}\end{aligned}\  \xmapsto{\ \beta_{\sigma_1}\ } \ 
 \begin{xy}
   (2,1.5); (4,3.5)**\dir{-}, 
   (0,3.5); (4,-.5)**\dir{-} ?(.5)*!RU{\scriptstyle \beta(e_1)}; 
   (4,-.5); (8,3.5)**\dir{-} ?(.2)*!LU{\scriptstyle \beta(e_2)}, 
   (4,-.5); (4,-2.5)**\dir{-}, 
   (0 ,4.7)*{\scriptstyle 2}, 
   (4,4.7)*{\scriptstyle 1}, 
   (8,4.7)*{\scriptstyle 3}, 
   (2,1.5)*{\scriptstyle \bullet},
   (4,-.5)*{\scriptstyle \bullet},
 \end{xy}  \qquad \qquad \qquad 
 \begin{aligned}\begin{xy} 
  (0,-2.5)*+UR{\scriptstyle 1}="a", 
  (3.75,3.75)*+UR{\scriptstyle 2}="b", 
  (7.5,-2.5)*+UR{\scriptstyle 3}="c", 
  "a";"b"**\dir{-}?>*\dir{>} ?(.4)*!RD{\scriptstyle e_1}, 
  "b";"c"**\dir{-}?>*\dir{>} ?(.5)*!LD{\scriptstyle e_2} 
 \end{xy}\end{aligned}\  \xmapsto{\ \beta_{\sigma_2}\ } \  
 \begin{xy}
   (2,1.5); (4,3.5)**\dir{-}, 
   (0,3.5); (4,-.5)**\dir{-} ?(.8)*!RU{\scriptstyle \beta(e_1)}; 
   (4,-.5); (8,3.5)**\dir{-} ?(.2)*!LU{\scriptstyle \beta(e_2)}, 
   (4,-.5); (4,-2.5)**\dir{-}, 
   (0 ,4.7)*{\scriptstyle 1}, 
   (4,4.7)*{\scriptstyle 3}, 
   (8,4.7)*{\scriptstyle 2}, 
   (4,-.5)*{\scriptstyle \bullet}, 
 \end{xy} $$
 In the first example, $\text{sgn}\bigl(\beta_{\sigma_1}(e_1)\bigr) = -1$ and 
 $\text{sgn}\bigl(\beta_{\sigma_1}(e_2)\bigr) = 1$. 
 In the second example, $\text{sgn}\bigl(\beta_{\sigma_2}(e_1)\bigr) = 1$ and 
 $\text{sgn}\bigl(\beta_{\sigma_2}(e_2)\bigr) = -1$. 
 The associated $\sigma$-configuration pairings are
 $\langle G,\, T\rangle_{\sigma_1} = -1$ and 
 $\langle G,\, T\rangle_{\sigma_2} = 0$.
\end{example}

\begin{definition}\label{D:configuration pairing}
On monomials $\gamma=(G,l_G)\in\mathrm{Gr}(V^*)$ and 
$\tau=(T,l_T)\in\mathrm{Tr}(V)$ 
let
$$\bigl\langle\gamma,\,\tau\bigr\rangle \ = 
\sum_{\sigma:\mathrm{V}(G)\xrightarrow{\cong}\mathrm{L}(T)} 
	\left(\langle G,\,T\rangle_\sigma\, 
	\prod_{v\in \mathrm{V}(G)} \Bigl\langle l_G(v),\,l_T\bigl(\sigma(v)\bigr)\Bigr\rangle\right)$$
where $\sum_\sigma$ is a sum over all isomorphisms $\sigma:\mathrm{Vertices}(G)\to \mathrm{Leaves}(T)$
and $\prod_v$ is a product over all vertices of $G$.  If there are no isomorphisms $\sigma$, then 
$\langle\gamma,\,\tau\rangle = 0$.  The configuration pairing is $\langle\,,\,\rangle$ extended to 
$\mathrm{Gr}(V^*)\times\mathrm{Tr}(V)$ by multilinearity.
\end{definition}

\begin{example}
 The configuration pairing
 $\left\langle 
 \begin{aligned}
	 \graphpp{b^*}{a^*}{b^*}
 \end{aligned},\ 
 \begin{aligned}\begin{xy}   
   (1.5,1.5); (3,3)**\dir{-}, 
   (0,3); (3,0)**\dir{-};
   (6,3)**\dir{-},   
   (3,0); (3,-1.5)**\dir{-}, 
   (-.4,4.2)*{\scriptstyle a}, 
   (3.2,4.2)*{\scriptstyle b},
   (6.8,4.2)*{\scriptstyle b}, 
 \end{xy}\end{aligned}\right\rangle = -2.$
 The isomorphisms 
$$
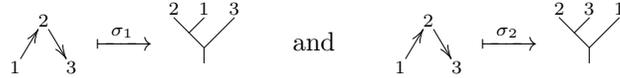
\begin{aligned}\begin{xy} 
  (0,-2.5)*+UR{\scriptstyle 1}="a", 
  (3.75,3.75)*+UR{\scriptstyle 2}="b", 
  (7.5,-2.5)*+UR{\scriptstyle 3}="c", 
  "a";"b"**\dir{-}?>*\dir{>}, 
  "b";"c"**\dir{-}?>*\dir{>} 
 \end{xy}\end{aligned}\  \xmapsto{\ {\sigma_1}\ } \ 
 \begin{xy}
   (2,1.5); (4,3.5)**\dir{-}, 
   (0,3.5); (4,-.5)**\dir{-}, 
   (4,-.5); (8,3.5)**\dir{-}, 
   (4,-.5); (4,-2.5)**\dir{-}, 
   (0 ,4.7)*{\scriptstyle 2}, 
   (4,4.7)*{\scriptstyle 1}, 
   (8,4.7)*{\scriptstyle 3}, 
 \end{xy}  \qquad \text{and} \qquad 
 \begin{aligned}\begin{xy} 
  (0,-2.5)*+UR{\scriptstyle 1}="a", 
  (3.75,3.75)*+UR{\scriptstyle 2}="b", 
  (7.5,-2.5)*+UR{\scriptstyle 3}="c", 
  "a";"b"**\dir{-}?>*\dir{>}, 
  "b";"c"**\dir{-}?>*\dir{>} 
 \end{xy}\end{aligned}\  \xmapsto{\ \sigma_2\ } \  
 \begin{xy}
   (2,1.5); (4,3.5)**\dir{-}, 
   (0,3.5); (4,-.5)**\dir{-},
   (4,-.5); (8,3.5)**\dir{-}, 
   (4,-.5); (4,-2.5)**\dir{-}, 
   (0 ,4.7)*{\scriptstyle 2}, 
   (4,4.7)*{\scriptstyle 3}, 
   (8,4.7)*{\scriptstyle 1}, 
 \end{xy} $$ 
 between $\mathrm{Vertices}(G)$ and $\mathrm{Leaves}(T)$ are the only two which will give 
 $\prod_v\bigl\langle l_G(v),\,l_T(\sigma(v))\bigr\rangle \neq 0$.
 These pair
 $\langle G,\, T\rangle_{\sigma_1} = -1$ and $\langle G,\,T\rangle_{\sigma_2}=-1$. 
 \end{example}

\begin{example}
 The configuration pairing
 $\left\langle 
 \begin{aligned}
	 \graphpp{a^*}{b^*}{b^*}
 \end{aligned},\ 
 \begin{aligned}\begin{xy}   
   (1.5,1.5); (3,3)**\dir{-}, 
   (0,3); (3,0)**\dir{-};
   (6,3)**\dir{-},   
   (3,0); (3,-1.5)**\dir{-}, 
   (-.4,4.2)*{\scriptstyle a}, 
   (3.2,4.2)*{\scriptstyle b},
   (6.8,4.2)*{\scriptstyle b}, 
 \end{xy}\end{aligned}\right\rangle = 1$.
 The isomorphisms 
$$
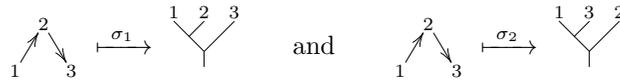
\begin{aligned}\begin{xy} 
  (0,-2.5)*+UR{\scriptstyle 1}="a", 
  (3.75,3.75)*+UR{\scriptstyle 2}="b", 
  (7.5,-2.5)*+UR{\scriptstyle 3}="c", 
  "a";"b"**\dir{-}?>*\dir{>}, 
  "b";"c"**\dir{-}?>*\dir{>} 
 \end{xy}\end{aligned}\  \xmapsto{\ {\sigma_1}\ } \ 
 \begin{xy}
   (2,1.5); (4,3.5)**\dir{-}, 
   (0,3.5); (4,-.5)**\dir{-}, 
   (4,-.5); (8,3.5)**\dir{-}, 
   (4,-.5); (4,-2.5)**\dir{-}, 
   (0 ,4.7)*{\scriptstyle 1}, 
   (4,4.7)*{\scriptstyle 2}, 
   (8,4.7)*{\scriptstyle 3}, 
 \end{xy}  \qquad \text{and} \qquad 
 \begin{aligned}\begin{xy} 
  (0,-2.5)*+UR{\scriptstyle 1}="a", 
  (3.75,3.75)*+UR{\scriptstyle 2}="b", 
  (7.5,-2.5)*+UR{\scriptstyle 3}="c", 
  "a";"b"**\dir{-}?>*\dir{>}, 
  "b";"c"**\dir{-}?>*\dir{>} 
 \end{xy}\end{aligned}\  \xmapsto{\ \sigma_2\ } \  
 \begin{xy}
   (2,1.5); (4,3.5)**\dir{-}, 
   (0,3.5); (4,-.5)**\dir{-},
   (4,-.5); (8,3.5)**\dir{-}, 
   (4,-.5); (4,-2.5)**\dir{-}, 
   (0 ,4.7)*{\scriptstyle 1}, 
   (4,4.7)*{\scriptstyle 3}, 
   (8,4.7)*{\scriptstyle 2}, 
 \end{xy} $$ 
 between $\mathrm{Vertices}(G)$ and $\mathrm{Leaves}(T)$ are the only two which will give 
 $\prod_v\bigl\langle l_G(v),\,l_T(\sigma(v))\bigr\rangle \neq 0$.
 These pair $\langle G,\, T\rangle_{\sigma_1} = 1$ and $\langle G,\,T\rangle_{\sigma_2}=0$. 
\end{example}

\begin{example}
 The configuration pairing
 $\left\langle 
 \begin{aligned}
	 \graphpp{a^*}{b^*}{b^*}
 \end{aligned},\ 
 \begin{aligned}\begin{xy}   
   (1.5,1.5); (3,3)**\dir{-}, 
   (0,3); (3,0)**\dir{-};
   (6,3)**\dir{-},   
   (3,0); (3,-1.5)**\dir{-}, 
   (-.4,4.2)*{\scriptstyle b}, 
   (3.2,4.2)*{\scriptstyle b},
   (6.8,4.2)*{\scriptstyle a}, 
 \end{xy}\end{aligned}\right\rangle = 0$.
 The isomorphisms 
$$\begin{aligned}\begin{xy} 
  (0,-2.5)*+UR{\scriptstyle 1}="a", 
  (3.75,3.75)*+UR{\scriptstyle 2}="b", 
  (7.5,-2.5)*+UR{\scriptstyle 3}="c", 
  "a";"b"**\dir{-}?>*\dir{>}, 
  "b";"c"**\dir{-}?>*\dir{>} 
 \end{xy}\end{aligned}\  \xmapsto{\ {\sigma_1}\ } \ 
 \begin{xy}
   (2,1.5); (4,3.5)**\dir{-}, 
   (0,3.5); (4,-.5)**\dir{-}, 
   (4,-.5); (8,3.5)**\dir{-}, 
   (4,-.5); (4,-2.5)**\dir{-}, 
   (0 ,4.7)*{\scriptstyle 2}, 
   (4,4.7)*{\scriptstyle 3}, 
   (8,4.7)*{\scriptstyle 1}, 
 \end{xy}  \qquad \text{and} \qquad 
 \begin{aligned}\begin{xy} 
  (0,-2.5)*+UR{\scriptstyle 1}="a", 
  (3.75,3.75)*+UR{\scriptstyle 2}="b", 
  (7.5,-2.5)*+UR{\scriptstyle 3}="c", 
  "a";"b"**\dir{-}?>*\dir{>}, 
  "b";"c"**\dir{-}?>*\dir{>} 
 \end{xy}\end{aligned}\  \xmapsto{\ \sigma_2\ } \  
 \begin{xy}
   (2,1.5); (4,3.5)**\dir{-}, 
   (0,3.5); (4,-.5)**\dir{-},
   (4,-.5); (8,3.5)**\dir{-}, 
   (4,-.5); (4,-2.5)**\dir{-}, 
   (0 ,4.7)*{\scriptstyle 3}, 
   (4,4.7)*{\scriptstyle 2}, 
   (8,4.7)*{\scriptstyle 1}, 
 \end{xy} $$ 
 between $\mathrm{Vertices}(G)$ and $\mathrm{Leaves}(T)$ are the only two which will give 
 $\prod_v\bigl\langle l_G(v),\,l_T(\sigma(v))\bigr\rangle \neq 0$.
 These pair $\langle G,\, T\rangle_{\sigma_1} = -1$ and $\langle G,\,T\rangle_{\sigma_2}=1$. 
\end{example}

From \cite{SiWa07}, the configuration pairing vanishes on
the ideal of $\mathrm{Tr}(V)$ generated by 
the anti-symmetry and Jacobi identities and on the coideal of $\mathrm{Gr}(V)$
generated by the arrow-reversing and Arnold identities.  Thus the configuration
pairing descends to a pairing between 
$\mathrm{E}(V^*)$ and $\mathrm{L}(V)$.  
Furthermore we have the following.

\begin{theorem}[3.11 of \cite{SiWa07}]\label{T:perfect pair}
	Let $V$ be a finite dimensional vector space over a field of characterisic zero.
	The configuration pairing of $\mathrm{L}V$ and
	$\mathrm{E}V^*$ is a perfect pairing.
\end{theorem}

\begin{theorem}[3.14 of \cite{SiWa07}]\label{T:bracket/cobracket}
	Given $\gamma \in \mathrm{E}V^*$ and $[\tau_1, \tau_2] \in \mathrm{L}V$, 
	$$\bigl\langle \gamma,\, [\tau_1, \tau_2]\bigr\rangle =
	 \sum_i \bigl\langle \alpha_i,\, \tau_1\bigr\rangle \,
			\bigl\langle \beta_i,\, \tau_2\bigr\rangle$$
	where $]\gamma[ \ = \sum_i \alpha_i\otimes\beta_i$.
\end{theorem}

\begin{remark}
	One corollary of the current work is that we may remove the ``characteristic zero''
	assumption from Theorem~\ref{T:perfect pair} using a dimension argument.
\end{remark}

For more detail on the foundations and interpretations of the configuration pairing,
see \cite{Walt10.2}.

\section{Simple Words}\label{S:B}

Given a finite dimensional vector space $V$, 
choose an ordered basis $v_1<\dots<v_d$ and write $\mathcal{A}$ for
the set of all finite words written using the alphabet $\{v_1,\dots,v_d\}$.  
Let $\mathcal{B}$ be the set of (associative) Lyndon-Shirshov words in $\mathcal{A}$.  
Explicitly, the ordering of its alphabet 
induces a lexicographical ordering on the words $\mathcal{A}$. 
Let $\mathcal{B}$ be the collection of words $\omega \in \mathcal{A}$ where 
$\omega$ has smaller ordering than any of its cyclic permutations.

\begin{example}
	For simplicity, write 1 for $v_1$, 2 for $v_2$, etc.
	\begin{itemize}
		\item $112 \in \mathcal{B}$ but not $121$ or $211$.
		\item $111122, 111212 \in \mathcal{B}$ but not $112112$.
	\end{itemize}
\end{example}

The following proposition is classical, and is proven by a simple counting argument.

\begin{proposition}\label{P:Witt}
	$\mathcal{B}$ satisfies the Witt formula:
	$$\#\{\text{\rm elements of length }n\} = \frac{1}{n} \sum_{m|n} \mu(m)\, d^{\,n/m}$$
	where $\mu$ is the M\"obius function. 
\end{proposition}
%\begin{proof}
%	Consider equivalence classes of length $n$ words modulo cyclic permutation.  Each 
%	equivalence class has a unique element of minimal ordering except for
%	those of the form $[\omega\omega\dots\omega]$ i.e. $[\omega^m]$ where $m|n$.  (Note that
%	cyclic permutations of $\omega^m$ still have this form.)  There are $d^n$ words of 
%	length $n$ and there are $d^{\,n/m}$ 
%	words of the form $\omega^m$ for each $m|n$.  Removing words of the form $\omega^m$, 
%	using $\mu(m)$ to account for double counting, leads to the Witt formula. 
%\end{proof}

\begin{definition}
A simple word in an ordered alphabet is 
$\overbrace{x\dots x}^ky_1\dots y_\ell = x^ky_1\dots y_\ell$ where 
$x \neq y_i$ for all $i$ (and $k, \ell>0$).  
Two simple words are compatible if their initial letters
are the same.  
%We give a new order on compatible simple words as follows.
%Let $x^ky_1\dots y_\ell < x^rz_1\dots z_s$ either if $k > r$ or if 
%$k = r$ and $y_1\dots y_\ell < z_1\dots z_s$ in the 
%{\it deg-lex} order.
Compatible simple words are ordered lexicographically via the
ordering of their alphabet.
\end{definition}

Let $\mathcal{A}_0 \subset \mathcal{A}$ be the subset of singleton words and
$\mathcal{A}_1 \subset \mathcal{A}$ be the subset of simple words. 
Recursively define $\mathcal{A}_m\subset \mathcal{A}$ for 
$m>1$ as the simple words in an ordered 
alphabet of compatible words in $\mathcal{A}_{m-1}$.  
%Note that elements of $\mathcal{A}_m$ 
%are compatible if their initial $\mathcal{A}_{m-1}$-subwords are identical.
%Compatible elements of $\mathcal{A}_m$ are ordered by
%$\omega^k\psi_1\dots\psi_\ell < \omega^r\phi_1\dots\phi_s$ either if $k>r$ or if
%$k=r$ and $\psi_1\dots\psi_\ell < \phi_1\dots\phi_s$ in the ordering induced
%from that on $\mathcal{A}_{m-1}$ by viewing strings of elements of $\mathcal{A}_{m-1}$ 
%as numbers.  
Note that to be in $\mathcal{A}_m$, a word must be length at least $2^m$.

\begin{example}\label{E:A_n}
	For clarity, in the following we will
	insert parenthesis in the $\mathcal{A}_2$, $\mathcal{A}_3$, and $\mathcal{A}_4$ examples
	to indicate subwords from lower levels.  
	Suppose the ordered alphabet $\mathcal{A}_0$ is $\{1<2<\cdots<9\}$. 
	\begin{itemize}
		\item $112 \ < \ 12 \ < \ 122 \ < \ 13 \in \mathcal{A}_1$.
		\item $(112)(112)(122) \ < \ (112)(12) \in \mathcal{A}_2$.
		\item $\bigl((112)(112)(12)\bigr)\bigl((112)(1122)\bigr) \ < \  
			\bigl((112)(112)(12)\bigr)\bigl((112)(12)\bigr)\bigl( (112)(12)\bigr) \in \mathcal{A}_3$.
		\item $\Bigl(\bigl( (112)(112)(12)\bigr)\bigl( (112)(12) \bigr)\bigl( (112)(122) \bigr)\Bigr)
			\Bigl(\bigl( (112)(112)(12)\bigr)\bigl( (112)(122)\bigr)\Bigr) \in \mathcal{A}_4$.
	\end{itemize}
%	As the $\mathcal{A}_1$ example above illustrates, the ordering on compatible words is neither the 
%	lexicographical ordering nor the 
%	{\it deg-lex} ordering.
\end{example}

It is quick to check that the $\mathcal{A}_m$ are disjoint and 
$\mathcal{A} = \cup_m \mathcal{A}_m$.   The key fact is that decomposition of a  
word into the form $\omega^k\psi_1\dots\psi_\ell$ with $\omega,\psi_i\in\mathcal{A}_{m-1}$ 
compatible is unique.
When we write $\omega^k\psi_1\dots\psi_\ell\in \mathcal{A}_m$ our implication will be that 
the presented decomposition is the unique decomposition into compatible words of $\mathcal{A}_{m-1}$.
For each $m$
define $\mathcal{B}_m = \mathcal{B}\cap\mathcal{A}_m$.  
The examples from \ref{E:A_n} are all members of $\mathcal{B}$.
The ordering of compatible words is chosen so that the following is true.

\begin{lemma}\label{L:order of B}
$\omega^k\psi_1\dots\psi_\ell \in \mathcal{B}_m$ 
if and only if $\omega, \psi_i\in \mathcal{B}_{m-1}$ 
and $\omega < \psi_i$ for all $i$ 
(in the ordering of compatible words of $\mathcal{A}_{m-1}$).
\end{lemma}
\begin{proof}[Proof Sketch]
	It is enough to show that
	the ordering on compatible words of $\mathcal{A}_{m-1}$ (as words in an alphabet
	of compatible words of $\mathcal{A}_{m-2}$) coincides with their 
	lexicographical ordering as words of $\mathcal{A}$.	  Induct.
%	given $\{\phi_i\}_{i=1}^k\subset\mathcal{B}_{m-1}$ compatible with 
%	$\phi_1\neq \phi_j$, then 
%	$\phi_1\cdots\phi_k < \phi_j\cdots\phi_k\phi_1\cdots \phi_{j-1}$ in the lexicographical 
%	ordering on words of $\mathcal{A}$
%	if and only if $\phi_1 < \phi_j$ in the ordering on compatible words of $\mathcal{B}_{m-1}$.
%
%	In the ``if'' direction, $\omega<\psi_i$ follows from the entire word being 
%	minimal in its cyclic ordering class, and
%	$\omega,\psi_i\in \mathcal{B}_{m-1}$ follows from this as well as the 
%	induction and the implicit statement that $\omega$ 
%	and $\psi_i$ are compatible.  
\end{proof}

\section{The Configuration Basis}\label{S:main}

Define maps $\mathcal{L}:\mathcal{B}\longrightarrow \mathrm{Tr}(V)$
and $\mathcal{G}:\mathcal{B} \longrightarrow \mathrm{Gr}(V^*)$
recursively as follows. 
On $\mathcal{B}_0$, $\mathcal{L}v_i = v_i$ and $\mathcal{G}v_i = v_i^*$ (the graph with a single
vertex labeled by $v_i^*$).  Call the single vertex of $\mathcal{G}v_i$ the pivot vertex. 
For $\omega^k\psi_i\dots\psi_\ell \in \mathcal{B}_m$ define
\begin{align*}
\mathcal{L}:
	\omega^k\psi_1\dots\psi_\ell &\longmapsto
	[[[[[[\mathcal{L}\psi_1,\overbrace{\mathcal{L}\omega], \dots], \mathcal{L}\omega}^k],
	\mathcal{L}\psi_2], \dots], \mathcal{L}\psi_\ell] \\
\mathcal{G}: 
	\omega^k\psi_1\dots\psi_\ell &\longmapsto
	\begin{aligned}\begin{xy}                           
		(0,3)*+UR{\scriptstyle \mathcal{G}\omega}="x1",
		(-6,-2)*+UR{\scriptstyle \cdots}="x2",
		(-12,3)*+UR{\scriptstyle \mathcal{G}\omega}="x3",
		(8,-2)*+UR{\scriptstyle \mathcal{G}\psi_1}="a",  
		(16,3)*+UR{\scriptstyle \mathcal{G}\psi_2}="b",     
		(22,-2)*+UR{\scriptstyle \cdots}="c", 
		(28,3)*+UR{\scriptstyle \mathcal{G}\psi_\ell}="d",  
		"a";"x1"**\dir{-}?>*\dir{>},
		"x1";"x2"**\dir{-}?>*\dir{>},
		"x2";"x3"**\dir{-}?>*\dir{>},
		"a";"b"**\dir{-}?>*\dir{>},    
		"b";"c"**\dir{-}?>*\dir{>},   
		"c";"d"**\dir{-}?>*\dir{>}   
	\end{xy}\end{aligned}	.
%	\begin{aligned}\begin{xy}                           
%		(-2,-6)*+UR{\scriptstyle \mathcal{G}\omega}="x1",
%		(-2,-1)*{\vdots},
%		(-2,2)*+UR{\scriptstyle \mathcal{G}\omega}="x3",
%		(8,-2)*+UR{\scriptstyle \mathcal{G}\psi_1}="a",  
%		(16,3)*+UR{\scriptstyle \mathcal{G}\psi_2}="b",     
%		(22,-2)*+UR{\scriptstyle \cdots}="c", 
%		(28,3)*+UR{\scriptstyle \mathcal{G}\psi_\ell}="d",  
%		"a";"x1"**\dir{-}?>*\dir{>},
%		"a";(1,-2)**\dir{-}?>*\dir{>},
%		"a";"x3"**\dir{-}?>*\dir{>},
%		"a";"b"**\dir{-}?>*\dir{>},    
%		"b";"c"**\dir{-}?>*\dir{>},   
%		"c";"d"**\dir{-}?>*\dir{>}   
%	\end{xy}\end{aligned}	.
\end{align*}
where the subgraph $\mathcal{G}\omega$ appears $k$ times and
the arrows connecting the subgraphs above connect their pivot vertices.
The pivot vertex of the above graph is inherited from $\mathcal{G}\psi_1$.

\begin{example} 
	Following are some examples of $\mathcal{L}\omega$ and $\mathcal{G}\omega$ for 
	$\omega \in \mathcal{B}_1, \mathcal{B}_2, \mathcal{B}_3$.  
	For clarity we will draw the larger bracket expressions
	also as trees when writing $\mathcal{L}$ and we will neglect ${}^*$ 
	when writing $\mathcal{G}$.
	\begin{itemize}
		\item $\mathcal{L}12 = [2,1]$  and $\mathcal{L}112 = \bigl[ [2,1], 1\bigr]$ and
			$\mathcal{L}11234 = \biggl[\Bigl[\bigl[[2,1],1\bigr],3\Bigr],4\biggr] $.

			\noindent
			$\mathcal{G}12 = 
				\begin{aligned}\begin{xy}                           
					(0,2)*+UR{\scriptstyle 1}="x1",
					(4,-2)*+UR{\scriptstyle 2}="a",  
					"a";"x1"**\dir{-}?>*\dir{>},
				\end{xy}\end{aligned}$
			and
			$\mathcal{G}112 = 
				\begin{aligned}\begin{xy}                           
					(-4,-2)*+UR{\scriptstyle 1}="x1",
					(0,2)*+UR{\scriptstyle 1}="x2",
					(4,-2)*+UR{\scriptstyle 2}="a",  
					"x2";"x1"**\dir{-}?>*\dir{>},
					"a";"x2"**\dir{-}?>*\dir{>},
				\end{xy}\end{aligned}$
			and
			$\mathcal{G}11234 = 
				\begin{aligned}\begin{xy}                           
					(-4,-2)*+UR{\scriptstyle 1}="x1",
					(0,2)*+UR{\scriptstyle 1}="x2",
					(4,-2)*+UR{\scriptstyle 2}="a",  
					(8,2)*+UR{\scriptstyle 3}="b",     
					(12,-2)*+UR{\scriptstyle 4}="c",     
					"a";"x2"**\dir{-}?>*\dir{>},
					"x2";"x1"**\dir{-}?>*\dir{>},
					"a";"b"**\dir{-}?>*\dir{>},    
					"b";"c"**\dir{-}?>*\dir{>},    
				\end{xy}\end{aligned}$.

		\noindent
		\item $\mathcal{L}(112)(112)(13)(142) = 
			\Biggl[\biggl[\Bigl[[3,1],\bigl[[2,1],1\bigr]\Bigr],\bigl[ [2,1], 1\bigr]\biggr], 
				\bigl[[4,1],2\bigr]\Biggr] =
			\begin{aligned}\begin{xy}
				(2,0)*{\scriptstyle 3}="a",
				(6,0)*{\scriptstyle 1}="b",
				(4,-2);"a"**\dir{-},
				(4,-2);"b"**\dir{-},
				(10,0)*{\scriptstyle 2}="c",
				(14,0)*{\scriptstyle 1}="d",
				(16,-2)*{\scriptstyle 1}="e",
				(12,-2);"c"**\dir{-},
				(12,-2);"d"**\dir{-},
				(12,-2);(14,-4)**\dir{-},
				(14,-4);"e"**\dir{-},
				(14,-4);(9,-6)**\dir{-},
				(20,-2)*{\scriptstyle 2}="f",
				(24,-2)*{\scriptstyle 1}="g",
				(26,-4)*{\scriptstyle 1}="h",
				(22,-4);"f"**\dir{-},
				(22,-4);"g"**\dir{-},
				(22,-4);(24,-6)**\dir{-},
				(24,-6);"h"**\dir{-},
				(4,-2);(14,-10)**\dir{-},
				(24,-6);(14,-10)**\dir{-},
				(14,-10);(19,-14)**\dir{-},
				(19,-14);(34,-8)**\dir{-},
				(30,-4)*{\scriptstyle 4}="i",
				(34,-4)*{\scriptstyle 1}="j",
				(36,-6)*{\scriptstyle 2}="k",
				(32,-6);"j"**\dir{-},
				(34,-8);"i"**\dir{-},
				(34,-8);"k"**\dir{-}
			\end{xy}\end{aligned}$.

		\noindent
		$\mathcal{G}(112)(112)(13)(142) = 
			\begin{aligned}\begin{xy}
					(5,7)*+UR{\scriptstyle 1}="x11",
					(0,9)*+UR{\scriptstyle 1}="x31",
					(7,2)*+UR{\scriptstyle 2}="a1",  
					"a1";"x11"**\dir{-}?>*\dir{>},
					"x11";"x31"**\dir{-}?>*\dir{>},
					(-5,5)*+UR{\scriptstyle 1}="x12",
					(-10,7)*+UR{\scriptstyle 1}="x32",
					(-3,0)*+UR{\scriptstyle 2}="a2",  
					"a2";"x12"**\dir{-}?>*\dir{>},
					"x12";"x32"**\dir{-}?>*\dir{>},
					(17,0)*+UR{\scriptstyle 3}="b",
					(15,5)*+UR{\scriptstyle 1}="c",
					"b";"a1"**\dir{-}?>*\dir{>},
					"a1";"a2"**\dir{-}?>*\dir{>},
					"b";"c"**\dir{-}?>*\dir{>},
					(27,2)*+UR{\scriptstyle 4}="d",
					(25,7)*+UR{\scriptstyle 1}="e",
					(31,-2)*+UR{\scriptstyle 2}="f",
					"b";"d"**\dir{-}?>*\dir{>},
					"d";"e"**\dir{-}?>*\dir{>},
					"d";"f"**\dir{-}?>*\dir{>},
			\end{xy}\end{aligned}$.

		\item $\mathcal{L}\bigl( (12) (13)\bigr)\bigl( (12)(14)\bigr) =
			\Bigl[\bigl[[4,1], [2,1]\bigr], \bigl[[3,1], [2,1]\bigr]\Bigl] =
			\begin{aligned}\begin{xy}
				(0,0)*{\scriptstyle 4}="a",
				(4,0)*{\scriptstyle 1}="b",
				(8,0)*{\scriptstyle 2}="c",
				(12,0)*{\scriptstyle 1}="d",
				(16,0)*{\scriptstyle 3}="e",
				(20,0)*{\scriptstyle 1}="f",
				(24,0)*{\scriptstyle 2}="g",
				(28,0)*{\scriptstyle 1}="h",
				"a";(2,-2)**\dir{-},
				"b";(2,-2)**\dir{-},
				"c";(10,-2)**\dir{-},
				"d";(10,-2)**\dir{-},
				"e";(18,-2)**\dir{-},
				"f";(18,-2)**\dir{-},
				"g";(26,-2)**\dir{-},
				"h";(26,-2)**\dir{-},
				(2,-2);(6,-4)**\dir{-},
				(10,-2);(6,-4)**\dir{-},
				(18,-2);(22,-4)**\dir{-},
				(26,-2);(22,-4)**\dir{-},
				(6,-4);(14,-6)**\dir{-},
				(22,-4);(14,-6)**\dir{-}
			\end{xy}\end{aligned}$.

		\noindent
		$\mathcal{G}\bigl( (12) (13)\bigr)\bigl( (12)(14)\bigr) =
			\begin{aligned}\begin{xy}
				(0,0)*+UR{\scriptstyle 1}="000",
				(6,-3)*+UR{\scriptstyle 2}="001",
				(12,0)*+UR{\scriptstyle 1}="010",
				(18,-3)*+UR{\scriptstyle 3}="011",
				(4,-7)*+UR{\scriptstyle 1}="100",
				(10,-10)*+UR{\scriptstyle 2}="101",
				(16,-7)*+UR{\scriptstyle 1}="110",
				(22,-10)*+UR{\scriptstyle 4}="111",
				"111";"011"**\dir{-}?>*\dir{>},
				"111";"101"**\dir{-}?>*\dir{>},
				"111";"110"**\dir{-}?>*\dir{>},
				"101";"100"**\dir{-}?>*\dir{>},
				"011";"010"**\dir{-}?>*\dir{>},
				"011";"001"**\dir{-}?>*\dir{>},
				"001";"000"**\dir{-}?>*\dir{>},
			\end{xy}\end{aligned}$.
	\end{itemize}
\end{example}

Write $\mathcal{G}\mathcal{B}$ and $\mathcal{L}\mathcal{B}$ for the images of 
the set of Lyndon words 
under $\mathcal{G}$ and $\mathcal{L}$.

\begin{theorem}\label{T:main}
	Bracket expressions $\mathcal{L}\mathcal{B}$ and graph expressions 
	$\mathcal{G}\mathcal{B}$ are dual vector space bases for the free
	Lie algebra $\mathrm{L}V$
	and the cofree conilpotent
	Lie coalgebra $\mathrm{E}V^*$.
\end{theorem}
\begin{proof}
%	From Proposition~\ref{P:Witt} and the fact that $\mathcal{L}$, $\mathcal{G}$ are
%	injections to the set of bracket and graph expressions, we know that the
%	number of bracket and graph expressions in $\mathcal{L}\mathcal{B}$ and 
%	$\mathcal{G}\mathcal{B}$ of each length is equal to the dimensions of the 
%	vector subspaces of $\mathrm{L}V$ and
%	$\mathrm{E}V^*$ of expressions of those lengths.  
	Applying Proposition~\ref{P:Witt}, it is enough to 
	show that $\mathcal{G}\mathcal{B}$ and $\mathcal{L}\mathcal{B}$ pair perfectly
	to prove both that the sets $\mathcal{G}\mathcal{B}$ and $\mathcal{L}\mathcal{B}$
	are each independent, and thus bases, and also that they are dual bases.

	Let $\upsilon, \omega \in \mathcal{B}$, and suppose 
	$\langle \mathcal{G}\omega, \mathcal{L}\upsilon\rangle \neq 0$.
	Fix a bijection
	$\sigma:\mathrm{Vertices}(\mathcal{G}\omega) \xrightarrow{\ \cong\ }
	\mathrm{Leaves}(\mathcal{L}\upsilon)$
	so that its term in the configuration pairing is nonzero.
	Note that $\sigma$ induces a bijection between the letters of $\upsilon$ 
	and those of $\omega$ with repetition.
	Let $x$ be the minimal letter in $\upsilon$ and $\omega$.  This letter
	is also their initial letter, since they are words in $\mathcal{B}$.

	From the definition of $\mathcal{L}$, the innermost brackets of 
	$\mathcal{L}\upsilon$ are of the form $[y_{i,1},x]$ for some letters $y_{i,1}$.
	There must be corresponding edges in 
	$\mathcal{G}\omega$ between vertices labeled $x^*$ and $y^*_{i,1}$ corresponding under $\sigma$. 
	From the definition of 
	$\mathcal{G}$, these edges must be 
	$\begin{aligned}\begin{xy}                           
		(5,2)*{\scriptstyle x^*}="x3",
		(12,-2)*{\scriptstyle y^*_{i,1}}="a",  
		"a";"x3"**\dir{-}?>*\dir{>}
	\end{xy}\end{aligned}$.

	Further, $\mathcal{L}\upsilon$ breaks into a series of maximal length bracket
	expressions 
	$[ [ [ [ y_{i,1}, x],\dots x], y_{i,2}], \dots y_{i,\ell_i}]$ with 
	$x\neq y_{i,j}\in \mathcal{A}_0$.
	A nonzero pairing with $[ [ y_{i,1}, x], x]$ implies that there is an edge
	from the graph 
	$\begin{aligned}\begin{xy}                           
		(0,2)*{\scriptstyle x^*}="x3",
		(6,-2)*{\scriptstyle y^*_{i,1}}="a",  
		"a";"x3"**\dir{-}?>*\dir{>}
	\end{xy}\end{aligned}$	to a vertex labeled $x^*$ in 
	$\mathcal{G}\omega$.   Since non-$x^*$ vertices are connected to at most
	one $x^*$ in graphs $\mathcal{G}\omega$, this edge must be
	$\begin{aligned}\begin{xy}                           
		(0,-2)*+U{\scriptstyle x^*}="x2",
		(6,2)*+U{\scriptstyle x^*}="x3",  
		(12,-2)*{\scriptstyle y^*_{i,1}}="a",
		"a";"x3"**\dir{-}?>*\dir{>},
		"x3";"x2"**\dir{-}?>*\dir{>}
	\end{xy}\end{aligned}$.
	Continuing in this manner, the bracket expression 
	$[ [y_{i,1}, x], \dots x]$ must correspond to a subgraph
	$\begin{aligned}\begin{xy}                           
		(-6,2)*{\scriptstyle x^*}="x1",
		(0,-2)*+UR{\scriptstyle \cdots}="x2",
		(6,2)*+U{\scriptstyle x^*}="x3",
		(12,-2)*{\scriptstyle y^*_{i,1}}="a",  
		"x2";"x1"**\dir{-}?>*\dir{>},
		"x3";"x2"**\dir{-}?>*\dir{>},
		"a";"x3"**\dir{-}?>*\dir{>}
	\end{xy}\end{aligned}	$.  
	
	The next bracket, with $y_{i,2}$, implies that there
	must be an edge from this subgraph to a vertex labeled $y_{i,2}^*$ 
	of $\mathcal{G}\omega$.
	From the structure of $\mathcal{G}$, the edge cannot connect to an 
	$x^*$, so it must connect to $y_{i,1}^*$.
	The next bracket, with $y_{i,3}$, implies that there is an edge from this subgraph to 
	the corresponding vertex $y_{i,3}^*$ in $\mathcal{G}\omega$.  
	If the edge came from $y_{i,1}^*$, then the structure of $\mathcal{G}\omega$ implies that one of the 
	vertices $y_{i,2}^*$ and $y_{i,3}^*$ would also have an edge to a vertex $x^*$.  This is not 
	possible, because this vertex would correspond to an $x$ in some other maximal length bracket expression
	$[ [ [ [y_{j,1}, x],\dots x], y_{j,2}],\dots,y_{j,\ell_j}]$ of $\mathcal{L}\upsilon$, implying the presence 
	of a subgraph
	$\begin{aligned}\begin{xy}                           
		(-12,-2)*{\scriptstyle y^*_{i,k}}="c",
		(-6,2)*+U{\scriptstyle x^*}="x1",
		(0,-2)*+UR{\scriptstyle \cdots}="x2",
		(6,2)*+U{\scriptstyle x^*}="x3",
		(12,-2)*{\scriptstyle y^*_{j,1}}="a",  
		"c";"x1"**\dir{-},
		"x2";"x1"**\dir{-}?>*\dir{>},
		"x3";"x2"**\dir{-}?>*\dir{>},
		"a";"x3"**\dir{-}?>*\dir{>}
	\end{xy}\end{aligned}$ in $\mathcal{G}\omega$.

	The next bracket, with $y_{i,4}$, implies an edge in $\mathcal{G}\omega$ 
	from this subgraph to a corresponding
	$y_{i,4}^*$. The edge cannot connect to $x^*$ or $y_{i,1}^*$ for the reasons already stated and 
	it cannot connect to $y_{i,2}^*$ because non-pivot vertices in $\mathcal{G}\omega$ are at
	most bivalent.
	Continuing in this way, we get that each maximal length bracket expression in 
	$\mathcal{L}\upsilon$
	must correspond to a subgraph 
	$\begin{aligned}\begin{xy}                           
		(-12,2)*+U{\scriptstyle x^*}="x1",
		(-6,-2)*+UR{\scriptstyle \cdots}="x2",
		(0,2)*+U{\scriptstyle x^*}="x3",
		(8,-2)*{\scriptstyle y_{i,1}^*}="a",  
		(18,2)*+U{\scriptstyle y_{i,2}^*}="b",     
		(25,-2)*+UR{\scriptstyle \cdots}="c", 
		(32,2)*+U{\scriptstyle y_{i,\ell_i}^*}="d",  
		"x2";"x1"**\dir{-}?>*\dir{>},
		"x3";"x2"**\dir{-}?>*\dir{>},
		"a";"x3"**\dir{-}?>*\dir{>},
		"a";"b"**\dir{-}?>*\dir{>},    
		"b";"c"**\dir{-}?>*\dir{>},   
		"c";"d"**\dir{-}?>*\dir{>}   
	\end{xy}\end{aligned}	$ of $\mathcal{G}\omega$.  Both the bracket expression and 
	the subgraph expression above correspond to an $\mathcal{A}_1$ subword 
	$x^ky_{i,1}\dots y_{i,\ell_i}$ of
	$\upsilon$ and $\omega$.

	Thus $\sigma$ gives a bijection of $\mathcal{A}_1$ subwords of $\omega$
	and $\upsilon$.  Let $\xi$ be the minimal $\mathcal{A}_1$ subword of 
	$\upsilon$ and $\omega$.  By Lemma~\ref{L:order of B}, this $\mathcal{A}_1$ 
	subword is also their 
	initial $\mathcal{A}_1$ subword. 
	Continue by induction (take the $\mathcal{A}_1$ subwords as the alphabet and look 
	at bracket expressions
	and graphs of them, etc).  Since $\upsilon$ is finite, it is in $\mathcal{B}_m$ 
	for some finite $m$.  At this stage, we will have identified $\omega = \upsilon$.

	It remains only to show that $\langle \mathcal{G}\omega, \mathcal{L}\omega\rangle = 1$.
	But this is clear.  Recursively apply the calculation
	$\bigl\langle\mathcal{G}x^ky_1\dots y_\ell,\; \mathcal{L}x^ky_1\dots y_\ell\bigr\rangle = 1$
	using the fact, established above, that any bijection $\sigma$ yielding a nonzero 
	term in the configuration pairing $\langle \mathcal{G}\omega,\, \mathcal{L}\upsilon\rangle$ 
	must give a bijection of $A_m$ subwords of $\omega$ and $\upsilon$.
\end{proof}

\begin{definition}
	$\mathcal{L}\mathcal{B}$ is the configuration basis of the free Lie algebra $\mathrm{L}V$.
\end{definition}

Theorem~\ref{T:main} has the following corollary, 
similar to the situation for Hall bases of Lie
algebras.

\begin{corollary}\label{C:integer coefficients}
	Lie bracket expressions, when written in terms of the configuration basis, 
	have integer coefficients.
\end{corollary}
\begin{proof}
	The Lie bracket expression $\ell$ will have $\mathcal{L}\omega$ coefficient
	$\langle \mathcal{G}\omega,\, \ell\rangle$.
	The configuration pairing of a graph and tree is always an integer.
\end{proof}

\section{Examples and Computations}\label{S:examples}

For simplicity we focus on the vector subspaces of $\mathrm{L}V$ 
where $d$ basis elements are repeated
$n_1,\dots,n_d$ times.  
Classically, a
counting argument on $\mathcal{B}$ similar to Proposition~\ref{P:Witt}
recovers the fine Witt formula 
$$\mathrm{dim}\left\{
\begin{array}{l}
	\text{\rm Lie brackets of $v_1,\dots,v_d$} \\
	\text{\rm with repetitions $n_1,\dots,n_d$}
\end{array}
\right\} 
\ =	\ \frac{1}{n}\!\sum_{m|(n_1,\dots,n_d)} \!\mu(m)\, 
		\frac{(\frac{n}{m})!}{(\frac{n_1}{m})!\cdots(\frac{n_d}{m})!}$$
where $n = \sum_i n_i$.
For comparison with other Lie bases and for use in later examples, 
we record the configuration basis for 
some of these vector subspaces along with  
associated elements of $\mathcal{B}$.
In the examples below, suppose $V$ has basis $x < y < z$.

\begin{example}\label{E:2x2y2z}
	The vector subspace of bracket expressions of 
	$x, y, z$ with each repeated twice has basis given by
	the following fourteen elements.
\begin{itemize}
	\item 
		$xxyyzz \mapsto [ [ [ [ [y, x], x], y], z], z]$.

	\noindent
		The five words 
		$xxyzyz$, $xxyzzy$, $xxzyyz$, $xxzyzy$, and $xxzzyy$ also have $\mathcal{L}$ of this form.
	\item 
		$xyxyzz \mapsto [ [ [ [y, x], z], z], [y, x] ]$

	\noindent
		The three words 
		$xyxzyz$, $xyxzzy$, and $xzxzyy$ also have $\mathcal{L}$ of this form.

	\item 
		$xyyzxz \mapsto [ [z, x], [ [ [y, x], y], z] ]$ 

		\noindent
		$\mathcal{L}xyzyxz$ is similar.

	\item 
		$xyyxzz \mapsto [ [ [z, x], z], [ [y, x], y] ]$.

		\noindent 
		$\mathcal{L}xyzxzy$ is similar. 
\end{itemize}
\end{example}

%\begin{remark}
%	The previous example exhibits a generic property of the configuration basis for Lie algebra
%	subspaces where many basis elements occur with repetition.  
%	The configuration basis in this case consists of only  
%	a few different bracket shapes, with fixed positions for the lowest ordered generators.
%	For example 
%	$[ [ [ [ [ [\cdot, x], x], \cdot], \cdot], \cdot], \cdot]$,
%	$[ [ [ [\cdot, x], \cdot], \cdot], [\cdot, x] ]$, and 
%	$[ [ [\cdot, x], \cdot], [ [\cdot, x], \cdot] ]$ above.  For most of these bracket shapes,
%	the other basis elements will occur in all permutations (with the caveat, for bracket shapes 
%	corresponding to	
%	$\mathcal{L}\mathcal{B}_{\ge 2}$, that certain positions must have matching elements).  
%	In bracket shapes with symmetry (such as $[ [ [ \cdot, x], \cdot], [ [ \cdot, x], \cdot] ]$), 
%	only some permutations occur.
%\end{remark}

\begin{example}\label{E:3x4y}
The vector subspace of bracket expressions 
with three $x$ and four $y$ has basis 
given by the following five elements. 
\begin{itemize}
	\item $xxxyyyy \mapsto [ [ [ [ [ [y, x], x], x], y], y], y]$
	\item $xxyxyyy \mapsto [ [ [ [y, x], y], y], [ [y, x], x] ]$
	\item $xxyyxyy \mapsto [ [ [y, x], y], [ [ [y, x], x], y] ]$ 
	\item $xxyyyxy \mapsto [ [y, x], [ [ [ [y, x], x], y], y] ]$
	\item $xyxyxyy \mapsto [ [ [ [y, x], y], [y, x] ], [y, x] ]$
\end{itemize}
\end{example}

We may use the dual graph basis and the configuration pairing 
in order to write general Lie bracket expressions 
in terms of the configuration basis by applying Theorem~\ref{T:main}. 
The following should be compared with the rewriting algorithm of \cite[Ex. 4.13]{Reut93},
which systematically applies the anti-symmetry and Jacobi identites to
different parts of bracket expression in order to rewrite as a linear combination of 
Hall basis elements.  This can take some time, since generically each application of 
the Jacobi identity adds one more bracket expression to which the algorithm must be 
applied.

\begin{example}\label{E:2x2y2z conversion1}
	We will write the bracket expression 
	$\ell = [[[x, y], [y, z]], [x,z]]$ in terms of the basis given in Example~\ref{E:2x2y2z}.  
	First, note that $\ell$ must pair to zero with
	$\mathcal{G}xxyyzz = 
	\begin{aligned}\begin{xy}
		(-6,-1)*+U{\scriptstyle x^*}="x1",
		(0,1)*+U{\scriptstyle x^*}="x3",
		(6,-1)*+U{\scriptstyle y^*}="a",  
		(12,1)*+U{\scriptstyle y^*}="b",     
		(18,-1)*+U{\scriptstyle z^*}="c",     
		(24,1)*+U{\scriptstyle z^*}="d",     
		"x3";"x1"**\dir{-}?>*\dir{>},
		"a";"x3"**\dir{-}?>*\dir{>},
		"a";"b"**\dir{-}?>*\dir{>},    
		"b";"c"**\dir{-}?>*\dir{>},    
		"c";"d"**\dir{-}?>*\dir{>},    
	\end{xy}\end{aligned}$,
	because no single cut on this graph will remove a subgraph containing
	only an $x^*$ and $z^*$ vertex.  Similar reasoning shows that $\ell$ pairs to 
	zero with all $\mathcal{G}\omega$ except for 
	$\mathcal{G}xyyzxz$, $\mathcal{G}xyzyxz$, and $\mathcal{G}xzxzyy$.
	After removing the $\linep{z^*}{x^*}$ subgraph from these, we are left with 
	$\mathcal{G}xyyz$, $\mathcal{G}xyzy$, and $\mathcal{G}xzyy$.  Of these, 
	$\mathcal{G}xzyy =
	\begin{aligned}\begin{xy}
		(0,1)*+U{\scriptstyle x^*}="x1",
		(6,-1)*+U{\scriptstyle z^*}="a",  
		(12,1)*+U{\scriptstyle y^*}="b",     
		(18,-1)*+U{\scriptstyle y^*}="c",     
		"a";"x1"**\dir{-}?>*\dir{>},
		"a";"b"**\dir{-}?>*\dir{>},    
		"b";"c"**\dir{-}?>*\dir{>},    
	\end{xy}\end{aligned}$ must pair to zero 
	with $[ [x, y], [y, z]]$ because no single cut will remove a subgraph containing only
	a $y^*$ and $z^*$ vertex.  So
	$[[[x, y], [y, z]], [x,z]] = k_1 \mathcal{L}xyyzxz + k_2 \mathcal{L}xyzyxz$.  To find $k_1$ and
	$k_2$ we compute pairings
	$k_1 = \langle \mathcal{G}xyyzxz, [ [ [x, y], [y, z]], [x, z]]\rangle = -1$ and
	$k_2 = \langle \mathcal{G}xyzyxz, [ [ [x, y], [y, z]], [x, z]]\rangle = 1$.
\end{example}

\begin{example}\label{E:2x2y2z conversion2}
	We will write the bracket expression
	$\ell = [ [ [ [ z, y], x], x], [y, z]]$ in terms of the basis given in Example~\ref{E:2x2y2z}. 
	Resoning as before, we immediatley see that $\ell$ pairs to zero with  
	$\mathcal{G}xxyyzz$, $\mathcal{G}xxzzyy$, $\mathcal{G}xyxyzz$, $\mathcal{G}xzxzyy$, and 
	$\mathcal{G}xyyxzz$.  For the remaining graphs, we compute pairings.
	\begin{itemize}
		\item $\langle \mathcal{G}xxyzyz,\, [ [ [ [ z, y], x], x], [y, z]]\rangle = -1$.
		\item $\langle \mathcal{G}xxyzzy,\, [ [ [ [ z, y], x], x], [y, z]]\rangle = 1$.
		\item $\langle \mathcal{G}xxzyyz,\, [ [ [ [ z, y], x], x], [y, z]]\rangle = 1$.
		\item $\langle \mathcal{G}xxzyzy,\, [ [ [ [ z, y], x], x], [y, z]]\rangle = -1$.
		\item $\langle \mathcal{G}xyxzyz,\, [ [ [ [ z, y], x], x], [y, z]]\rangle = 2$.
		\item $\langle \mathcal{G}xyxzzy,\, [ [ [ [ z, y], x], x], [y, z]]\rangle = -2$.
		\item $\langle \mathcal{G}xyzxzy,\, [ [ [ [ z, y], x], x], [y, z]]\rangle = 2$.
		\item $\langle \mathcal{G}xyyzxz,\, [ [ [ [ z, y], x], x], [y, z]]\rangle = 2$.
		\item $\langle \mathcal{G}xyzyxz,\, [ [ [ [ z, y], x], x], [y, z]]\rangle = -2$.
	\end{itemize}
	Pairing computations may be done either by iterating Theorem~\ref{T:bracket/cobracket},
	or by applying Definition~\ref{D:configuration pairing} directly, 
	or by some combination of the two.
	Applying Theorem~\ref{T:main} we have the following. 
	\begin{align*}
		[ [ [ [ z, y], x], x], [y, z]] = &-\mathcal{L}xxyzyz + \mathcal{L}xxyzzy +\mathcal{L}xxzyyz - \mathcal{L}xxzyzy \\
		&+2\,\mathcal{L}xyxzyz -2\,\mathcal{L}xyxzzy  +2\,\mathcal{L}xyzxzy+2\,\mathcal{L}xyyzxz -2\,\mathcal{L}xyzyxz
	\end{align*}
\end{example}

\begin{example}\label{E:3x4y conversion}
	We will write the bracket expression
	$\ell = [ [ [ [ [ [ x, y], y], x], x], y], y]$ in terms of the basis given in Example~\ref{E:3x4y}.
	Note that $\mathcal{G}xyxyxyy$ must pair to 
	zero with $\ell$ since we cannot make two consecutive cuts from the graph 
	$\begin{aligned}\begin{xy}
		(4,5)*+UR{\scriptstyle x^*}="x11",
		(7,-1)*+U{\scriptstyle y^*}="a1",  
		"a1";"x11"**\dir{-}?>*\dir{>},
		(-6,8)*+UR{\scriptstyle x^*}="x12",
		(-3,2)*+U{\scriptstyle y^*}="a2",  
		"a2";"x12"**\dir{-}?>*\dir{>},
		(16,2)*+U{\scriptstyle y^*}="b",
		(13,8)*+UR{\scriptstyle x^*}="c",
		(22,-3)*+UR{\scriptstyle y^*}="d",
		"b";"a1"**\dir{-}?>*\dir{>},
		"a1";"a2"**\dir{-}?>*\dir{>},
		"b";"c"**\dir{-}?>*\dir{>},
		"b";"d"**\dir{-}?>*\dir{>},
	\end{xy}\end{aligned}$
%		$\begin{aligned}\begin{xy}
%					(2,3)*+U{\scriptstyle x^*}="x11",
%					(7,-1)*+U{\scriptstyle y^*}="a1",  
%					"a1";"x11"**\dir{-}?>*\dir{>},
%					(1,-3)*+U{\scriptstyle x^*}="x12",
%					(6,-6)*+U{\scriptstyle y^*}="a2",  
%					"a2";"x12"**\dir{-}?>*\dir{>},
%					(19,-3)*+U{\scriptstyle y^*}="b",
%					(14,2)*+U{\scriptstyle x^*}="c",
%					"b";"a1"**\dir{-}?>*\dir{>},
%					"b";"a2"**\dir{-}?>*\dir{>},
%					"b";"c"**\dir{-}?>*\dir{>},
%					(25,-6)*+U{\scriptstyle y^*}="d",
%					"b";"d"**\dir{-}?>*\dir{>},
%				\end{xy}\end{aligned}$
	each time removing a single $y^*$ vertex.  For the remaining graphs, we compute pairings.
	\begin{itemize}
		\item $\langle \mathcal{G}xxxyyyy,\, [ [ [ [ [ [ x, y], y], x], x], y], y]\rangle = -1$.
		\item $\langle \mathcal{G}xxyxyyy,\, [ [ [ [ [ [ x, y], y], x], x], y], y]\rangle = 1$.
		\item $\langle \mathcal{G}xxyyxyy,\, [ [ [ [ [ [ x, y], y], x], x], y], y]\rangle = 2$.
		\item $\langle \mathcal{G}xxyyyxy,\, [ [ [ [ [ [ x, y], y], x], x], y], y]\rangle = 1$.
	\end{itemize}
	Thus
	$[ [ [ [ [ [ x, y], y], x], x], y], y] = 
	-\mathcal{L}xxxyyyy + \mathcal{L}xxyxyyy + 2\,\mathcal{L}xxyyxyy + \mathcal{L}xxyyyxy$.

	Note that the difficulty of these pairing calculations is more closely related to 
	the number of repetitions 
	than to the length of the bracket expression.
\end{example}

\section{The Classical Lie Coalgebra Basis}\label{S:Lyndon}

Classically, the free Lie algebra $\mathrm{L}V$ is spanned as a vector space by bracket 
expressions of
the form $[ [y_1, y_2], \dots, y_n]$.  Similarly 
the cofree conilpotent Lie coalgebra $\mathrm{E}V^*$ is spanned by ``long graphs''  --
those of the form
	$\begin{aligned}\begin{xy}                           
		(0,1)*+U{\scriptstyle y_1^*}="x3",
		(6,-1)*{\scriptstyle y_2^*}="a",  
		(12,1)*+U{\scriptstyle y_3^*}="b",     
		(18,-2)*+UR{\scriptstyle \cdots}="c", 
		(24,1)*+U{\scriptstyle y_n^*}="d",  
		"x3";"a"**\dir{-}?>*\dir{>},
		"a";"b"**\dir{-}?>*\dir{>},    
		"b";"c"**\dir{-}?>*\dir{>},   
		"c";"d"**\dir{-}?>*\dir{>}   
	\end{xy}\end{aligned}	$.
Write $y_1^*|y_2^*|\cdots|y_n^*$ for such a graph.
From \cite[Prop. 3.21]{SiWa07}, the coideal of $\mathrm{E}V^*$
generated by arrow-reversing and Arnold expressions 
is spanned by shuffles 
$\displaystyle \sum_{\sigma\in k\text{-Shuffle}} w_{\sigma(1)}|\dots|w_{\sigma(n)}$ 
where $k$-Shuffles are the shuffles of
$(1,\dots,k)$ into $(k+1,\dots,n)$.
Thus, we recover a classical representation
of the conilpotent Lie coalgebra implied by the equivalence of Harrison 
homology and commutative Andre-Quillen homology. 
The induced coalgebra structure, already noted 
in \cite{ScSt85}, is merely the anti-commutative cut coproduct.
The induced configuration pairing with Lie algebras may be 
quickly computed by recursively applying Theorem~\ref{T:bracket/cobracket}.

In this context, the configuration pairing recovers an item of classical 
interest. 
Write $p:\mathrm{L}V\to \mathrm{U}\mathrm{L}V$ for the 
standard map from $\mathrm{L}V$ to its universal enveloping algebra.
Recall that the universal enveloping algebra of a free Lie algebra is 
canonically isomorphic to the free associative algebra $\mathrm{T}V$, 
and in this case
$p$ is given by $p(v_i) = v_i$ on generators and 
$p([\ell_1,\ell_2]) = p(\ell_1)p(\ell_2) - p(\ell_2)p(\ell_1)$.

\begin{proposition}\label{P:config as coeff}
	The configuration pairing $\bigl\langle y_1^*|y_2^*|\cdots|y_n^*,\, \ell\bigr\rangle$
	is equal to the coefficient of the word $y_1y_2\cdots y_n$ in $p(\ell)$.
\end{proposition}
\begin{proof}
	This follows formally from Theorem~\ref{T:bracket/cobracket}, the structure of $p$
	noted above,
	and the 
	fact that on generators
	$\langle v_i^*, v_j \rangle = \delta(i,j)$.
\end{proof}

Write $(-)^*$ for the map $\mathcal{A}\to \mathrm{E}V^*$ which reads
a word as a bar word, $(y_1y_2\dots y_n)^* = y_1^*|y_2^*|\cdots|y_n^*$.
Classically, Lyndon words are a multiplicative basis for the 
algebra of all words with shuffle product \cite{Radf79}.  Combined with 
\cite[Prop. 3.21]{SiWa07} this implies 
$\mathcal{B}^*$ is a vector space basis for $\mathrm{E}V^*$.
%In fact, $\mathcal{B}^*$ is the only known such basis of words, 
%which may account for its central role in the theory of free Lie algebras.
%A few short calculations with the configuration pairing show that there are 
%in general no dual bases consisting of 
%monomials in $\mathrm{L}V$ and dual monomials of bar expressions.

\begin{example}\label{E:3x4y bars1}
	We can use the configuration pairing with $\mathcal{B}^*$ 
	to recover the result of Example~\ref{E:3x4y conversion}.  
	In Figure~1 we give the portion of the pairing matrix of  
	$\mathcal{B}^*$ and $\mathcal{L}\mathcal{B}$ 
	for the basis recorded in Example~\ref{E:3x4y} as well as with the
	Lie bracket expression from Example~\ref{E:3x4y conversion}.
	To conserve space, we write $\mathcal{L}\omega_1,\dots,\mathcal{L}\omega_7$ 
	for the (lexicographically ordered) basis of 
	Example~\ref{E:3x4y} and $\ell = [ [ [ [ [ [x, y], y], x], x], y], y]$.
	For visual clarity we leave blank positions where the configuration
	pairing is zero.

	\begin{figure}[h]\label{F:3x4y bars1}
	\begin{tabular}{c || c c c c c | c}
		& $\mathcal{L}\omega_1$
			& $\mathcal{L}\omega_2$
			& $\mathcal{L}\omega_3$
			& $\mathcal{L}\omega_4$
			& $\mathcal{L}\omega_5$ 
			& $\ell$ \\
			\hline \hline
		$\omega_1^*$ & -1 &    &    &    &    &  1 \\
		$\omega_2^*$ &  3 &  1 &    &    &    & -2 \\
		$\omega_3^*$ &    & -3 &  1 &    &    & -1 \\
		$\omega_4^*$ &    &  3 & -2 &  1 &    &    \\
		$\omega_5^*$ &    &  6 & -2 &  2 & -1 &  4 \\
	\end{tabular}

	\vskip 5pt
	$$\ell = -\mathcal{L}\omega_1 + \mathcal{L}\omega_2 + 2\,\mathcal{L}\omega_3 + \mathcal{L}\omega_4$$
	\caption{Pairing with configuration basis from Example~\ref{E:3x4y}. }
	\end{figure}

\end{example}

\begin{example}\label{E:2x2y2z bars}
	In the same way, we recover the result of 
	Examples~\ref{E:2x2y2z conversion1} and \ref{E:2x2y2z conversion2}.
	In Figure~2, we give the portion of the pairing matrix of  
	$\mathcal{B}^*$ and $\mathcal{L}\mathcal{B}$ for the 
	(lexicographically ordered) basis recorded in Example~\ref{E:2x2y2z}.
	It is an exercise for the reader to finish this
	example by computing pairings with the brackets of Examples~\ref{E:2x2y2z conversion1}
	and \ref{E:2x2y2z conversion2}.

	\begin{figure}[h]\label{F:2x2y2z bars}
	\begin{tabular}{c || c c c c c c c c c c c c c c }
		& $\mathcal{L}\omega_1$
			& $\mathcal{L}\omega_2$
			& $\mathcal{L}\omega_3$
			& $\mathcal{L}\omega_4$
			& $\mathcal{L}\omega_5$
			& $\mathcal{L}\omega_6$
			& $\mathcal{L}\omega_7$
			& $\mathcal{L}\omega_8$
			& $\mathcal{L}\omega_9$
			& $\mathcal{L}\omega_{10}$ 
			& $\mathcal{L}\omega_{11}$ 
			& $\mathcal{L}\omega_{12}$ 
			& $\mathcal{L}\omega_{13}$ 
			& $\mathcal{L}\omega_{14}$ 
			 \\
			\hline \hline
		$\omega_1^*$ &
	1& & & & & & & & & & & & & 	 \\
		$\omega_2^*$ &
	 &1& & & & & & & & & & & & 	 \\
		$\omega_3^*$ &
	 & &1& & & & & & & & & & & 	 \\
		$\omega_4^*$ &
	 & & &1& & & & & & & & & & 	 \\
		$\omega_5^*$ &
	 & & & &1& & & & & & & & & 	 \\
		$\omega_6^*$ &
	 & & & & &1& & & & & & & & 	 \\
		$\omega_7^*$ &
	-2& & & & & &-1& & & & & & & 	 \\
		$\omega_8^*$ &
	 &-2& & & & & &-1& & & & & & 	 \\
		$\omega_9^*$ &
	 & &-2& & & & & &-1& & & & & 	 \\
		$\omega_{10}^*$ &
	 & & & & & &1&1&1&-1& & & & 	 \\
		$\omega_{11}^*$ &
	 & & & & & & &-1&-2&2&-1& & & 	 \\
		$\omega_{12}^*$ &
	 & & & & & & & 1& 2& & &-1& & 	 \\
		$\omega_{13}^*$ &
	 & & & & & &-2&-1& & & & 1&-1& 	 \\
		$\omega_{14}^*$ &
	 & & & & &-2& & & & & & & &-1	 \\
	\end{tabular}
	\caption{Pairing with configuration basis from Example~\ref{E:2x2y2z}.} 
\end{figure}
%DATA:
% Configuration basis (ordered):
%  [[[[[2,1],1],2],3],3] [[[[[2,1],1],3],2],3] [[[[[2,1],1],3],3],2] [[[[[3,1],1],2],2],3] [[[[[3,1],1],2],3],2]
%  [[[[[3,1],1],3],2],2] [[[[2,1],3],3],[2,1]] [[[[3,1],2],3],[2,1]] [[[[3,1],3],2],[2,1]] [[[3,1],3],[[2,1],2]] 
%  [[3,1],[[[2,1],2],3]] [[[3,1],2],[[2,1],3]] [[3,1],[[[2,1],3],2]] [[[[3,1],2],2],[3,1]]
% Lyndon basis (ordered):
%  112233 112323 112332 113223 113232 113322 121233 121323 121332 122133 122313 123132 123213 131322
% Pairing matrix:
%	 1	 0	 0	0	 0	 0	 0	 0	 0	 0	 0	 0	 0	0	
%	 0	 1	 0	0	 0	 0	 0	 0	 0	 0	 0	 0	 0	0	
%	 0	 0	 1	0	 0	 0	 0	 0	 0	 0	 0	 0	 0	0	
%	 0	 0	 0	1	 0	 0	 0	 0	 0	 0	 0	 0	 0	0	
%	 0	 0	 0	0	 1	 0	 0	 0	 0	 0	 0	 0	 0	0	
%	 0	 0	 0	0	 0	 1	 0	 0	 0	 0	 0	 0	 0	0	
%	-2	 0	 0	0	 0	 0	-1	 0	 0	 0	 0	 0	 0	0	
%	 0	-2	 0	0	 0	 0	 0	-1	 0	 0	 0	 0	 0	0	
%	 0	 0	-2	0	 0	 0	 0	 0	-1	 0	 0	 0	 0	0	
%	 0	 0	 0	0	 0	 0	 1	 1	 1	-1	 0	 0	 0	0	
%	 0	 0	 0	0	 0	 0	 0	-1	-2	 2	-1	 0	 0	0	
%	 0	 0	 0	0	 0	 0	 0	 1	 2	 0	 0	-1	 0	0	
%	 0	 0	 0	0	 0	 0	-2	-1	 0	 0	 0	 1	-1	0	
%	 0	 0	 0	0	 0	-2	 0	 0	 0	 0	 0	 0	 0	-1

\end{example}

The shape of the pairing matrix in the previous examples is no coincidence.

\begin{theorem}\label{P:lower triangular}
	If $\omega < \upsilon \in \mathcal{B}$ then 
	$\langle\omega^*,\,\mathcal{L}\upsilon\rangle =0$.
	Furthermore $\langle\omega^*,\,\mathcal{L}\omega\rangle = \pm 1$.
\end{theorem}

In our proof, we make use of the following computational proposition and
its corollary.  (To conserve space
below, we neglect marking $^*$ on graph vertex labels.)

\begin{proposition}\label{L:convert simple words}
	Modulo the Arnold and arrow-reversing relations, the following local identity holds.
	$$\displaystyle
	\begin{xy}
		(0,3)*{\scriptstyle y_1}="x1",
		(0,1)*{\scriptstyle \vdots},
		(0,-3)*{\scriptstyle y_k}="x3",
		(9,0)*+UR{\scriptstyle z}="y",
  (4,-5),{\ar@{. }@(l,l)(3,6)},
  (4,6),{\ar@{. }@(r,r)(3,-5)},
  ?!{"y";"y"+/va(0)/}="y1",
  ?!{"y";"y"+/va(60)/}="y2",
  ?!{"y";"y"+/va(30)/}="y3",
		"y";"x1" **\dir{-}?>*\dir{>},
		"y";(1,0) **\dir{-}?>*\dir{>},
		"y";"x3" **\dir{-}?>*\dir{>},
		"y";"y1" **\dir{-},
		"y";"y2" **\dir{-},
		"y";"y3" **\dir{-},
	\end{xy} \ = \ (-1)^k\,\sum_{\sigma\in \Sigma_k}
	\begin{xy}
		(-2,-2)*{\scriptstyle y_{\sigma(1)}}="x3",
		(2,2)*{\scriptstyle \cdots}="x2",
		(6,-2)*{\scriptstyle y_{\sigma(k)}}="x1",
		(10,2)*+UR{\scriptstyle z}="y",
  (4,-5),{\ar@{. }@(l,l)(3,6)},
  (4,6),{\ar@{. }@(r,r)(3,-5)},
  ?!{"y";"y"+/va(0)/}="y1",
  ?!{"y";"y"+/va(60)/}="y2",
  ?!{"y";"y"+/va(30)/}="y3",
		"x1";"y" **\dir{-}?>*\dir{>},
		"x2";"x1" **\dir{-}?>*\dir{>},
		"x3";"x2" **\dir{-}?>*\dir{>},
		"y";"y1" **\dir{-},
		"y";"y2" **\dir{-},
		"y";"y3" **\dir{-},
	\end{xy}	$$
\end{proposition}

\begin{corollary}\label{C:convert simple words}
	Modulo the Arnold and arrow-reversing relations, the following local identity holds.
	$$\displaystyle
	{\scriptstyle k}\biggl\{
	\hspace{-25pt}
	\begin{xy}
		(0,3)*{\scriptstyle x}="x1",
		(0,1)*{\scriptstyle \vdots},
		(0,-3)*{\scriptstyle x}="x3",
		(9,0)*+UR{\scriptstyle y}="y",
  (4,-5),{\ar@{. }@(l,l)(3,6)},
  (4,6),{\ar@{. }@(r,r)(3,-5)},
  ?!{"y";"y"+/va(0)/}="y1",
  ?!{"y";"y"+/va(60)/}="y2",
  ?!{"y";"y"+/va(30)/}="y3",
		"y";"x1" **\dir{-}?>*\dir{>},
		"y";(1,0) **\dir{-}?>*\dir{>},
		"y";"x3" **\dir{-}?>*\dir{>},
		"y";"y1" **\dir{-},
		"y";"y2" **\dir{-},
		"y";"y3" **\dir{-},
	\end{xy} \ = \ (-1)^k\,k!
	\begin{xy}
		(-3,-2)*{\scriptstyle x}="x3",
		(1,2)*{\scriptstyle \cdots}="x2",
		(5,-2)*{\scriptstyle x}="x1",
		(9,2)*+UR{\scriptstyle y}="y",
  (4,-5),{\ar@{. }@(l,l)(3,6)},
  (4,6),{\ar@{. }@(r,r)(3,-5)},
  ?!{"y";"y"+/va(0)/}="y1",
  ?!{"y";"y"+/va(60)/}="y2",
  ?!{"y";"y"+/va(30)/}="y3",
		"x1";"y" **\dir{-}?>*\dir{>},
		"x2";"x1" **\dir{-}?>*\dir{>},
		"x3";"x2" **\dir{-}?>*\dir{>},
		"y";"y1" **\dir{-},
		"y";"y2" **\dir{-},
		"y";"y3" **\dir{-},
	\end{xy}.	$$
\end{corollary}

\begin{proof}[Proof of Theorem~\ref{P:lower triangular}]
	Applying Theorem~\ref{T:main}, the first statement is equivalent to 
	$\omega^* = \sum_i c_i\, \mathcal{G}\upsilon_i$ where
	$\upsilon_i \le \omega$.  We apply the Arnold identity repeatedly, 
	neglecting arrows for simplicity.

	Repeated applications of the Arnold identity 
	and Corollary~\ref{C:convert simple words} convert 
	$\begin{aligned}\begin{xy}    
		(-2,2)*+UR{\scriptstyle \cdots}="x3",
		(6,-2)*+U{\scriptstyle y_{i,\ell_i}}="a",  
		(13,1)*+U{\scriptstyle x}="b",     
		(19,-2)*+UR{\scriptstyle \cdots}="c", 
		(25,1)*+U{\scriptstyle x}="d",  
		(32,-2)*+U{\scriptstyle y_{i+1,1}}="e",
		(40,2)*+UR{\scriptstyle \cdots}="f",
		"x3";"a"**\dir{-},
		"a";"b"**\dir{-},    
		"b";"c"**\dir{-},   
		"c";"d"**\dir{-},
		"d";"e"**\dir{-},
		"e";"f"**\dir{-}
	\end{xy}\end{aligned}	$
	to a linear combination of graphs of the form 
	$\begin{aligned}\begin{xy}    
		(-2,2)*+UR{\scriptstyle \cdots}="x3",
		(6,-2)*+U{\scriptstyle y_{i,\ell_i}}="a",  
		(10,2)*{\scriptstyle x}="a1",
		(5,4)*{\scriptstyle \cdots}="a2",
		(10,6)*{\scriptstyle x}="a3",
		(18,-2)*+U{\scriptstyle y_{i+1,1}}="b",
		(27,2)*+UR{\scriptstyle \cdots}="c",
		(14,2)*{\scriptstyle x}="b1",
		(20,4)*{\scriptstyle \cdots}="b2",
		(14,6)*{\scriptstyle x}="b3",
		"x3";"a"**\dir{-},
		"a";"b"**\dir{-},    
		"b";"c"**\dir{-},   
		"a";"a1"**\dir{-},
		"a1";"a2"**\dir{-},
		"a2";"a3"**\dir{-},
		"b";"b1"**\dir{-},
		"b1";"b2"**\dir{-},
		"b2";"b3"**\dir{-},
	\end{xy}\end{aligned}$ 
	with strings of $x$ vertices of varying lengths
	off of the $y_{i,\ell_i}$ and $y_{i+1,1}$ vertices. 
	Similarly, applying Arnold and Proposition~\ref{L:convert simple words} 
	converts the graph
	$\begin{aligned}\begin{xy}    
		(-3,2)*+UR{\scriptstyle \cdots}="x3",
		(6,-2)*+U{\scriptstyle y_{i,1}}="a",  
		(3,2)*{\scriptstyle x}="x11",
		(7,4)*{\scriptstyle \cdots}="x12",
		(13,1)*+U{\scriptstyle y_{i,2}}="b",     
		(19,-2)*+UR{\scriptstyle \cdots}="c", 
		(25,1)*+U{\scriptstyle y_{i,k}}="d",  
		(32,-2)*+U{\scriptstyle z}="e",
		(34,2)*{\scriptstyle x}="x21",
		(30,4)*{\scriptstyle \cdots}="x22",
		(39,2)*+UR{\scriptstyle \cdots}="f",
		"x3";"a"**\dir{-},
		"a";"b"**\dir{-},    
		"b";"c"**\dir{-},   
		"c";"d"**\dir{-},
		"d";"e"**\dir{-},
		"e";"f"**\dir{-},
		"a";"x11"**\dir{-},
		"x11";"x12"**\dir{-},
		"e";"x21"**\dir{-},
		"x21";"x22"**\dir{-},
	\end{xy}\end{aligned}	$
	(with either $z = y_{i,\ell_i}, k=\ell_i-1$ or $z=y_{i+1,1}, k=\ell_i$)
	to a linear combination of graphs of the form
	$$\begin{aligned}\begin{xy}
		(-8,-1)*{\scriptstyle \cdots}="ldots",
		(0,0)*+U{\scriptstyle y_{i,1}}="y1",
		(-3,4)*{\scriptstyle x}="yx",
		(-7,5)*{\scriptstyle \cdots}="yxdots",
		(4,-5)*{\scriptstyle y_{i,2}}="y2",
		(11,-8)*{\scriptstyle \cdots}="y3",
		(19,-10)*{\scriptstyle y_{i,j}}="y4",
		(20,0)*{\scriptstyle z}="z1",
		(26,1)*{\scriptstyle \cdots}="rdots",
		(17,4)*{\scriptstyle x}="zx",
		(13,5)*{\scriptstyle \cdots}="zxdots",
		(23,-4)*{\scriptstyle y_{i,\sigma(j+1)}}="z2",
		(30,-7)*{\scriptstyle \cdots}="z3",
		(38,-9)*{\scriptstyle y_{i,\sigma(k)}}="z4",
		"ldots";"y1" **\dir{-},
		"y1";"yx" **\dir{-},
		"yx";"yxdots" **\dir{-},
		"y1";"y2" **\dir{-},
		"y2";"y3" **\dir{-},
		"y3";"y4" **\dir{-},
		"y1";"z1" **\dir{-},
		"z1";"rdots" **\dir{-},
		"zx";"z1" **\dir{-},
		"zx";"zxdots" **\dir{-},
		"z1";"z2" **\dir{-},
		"z2";"z3" **\dir{-},
		"z3";"z4" **\dir{-},
	\end{xy}\end{aligned}$$
	where $1\le j\le k$ and $\sigma$ is some permutation of $\{j+1,\dots,k\}$ (if $j=1$ then
	$y_{i,1}$ has no ``tail'', if $j=k$ then $z$ has no ``tail'').
	Note in particular that $y_{i,1}\dots y_{i,j} \le y_{i,1}\dots y_{i,j}\dots y_{i,k}$.
	Applying these two steps across all of 
	$\omega$ from left to right, combining ``tails'' using 
	Proposion~\ref{L:convert simple words}, and then recursing over the 
	grading $\mathcal{A}_n$ 
	(letters, simple words, simple words of simple words, etc),
	rewrites $\omega^* = \sum_i c_i\,\mathcal{G}\upsilon_i$ where
	$\upsilon_i \le \omega$.

	To check $\langle \omega^*,\,\mathcal{L}\omega\rangle = \pm 1$,
	it is enough to follow through the Arnold identity applications above, keeping
	track of which graphs will eventually lead to $\mathcal{G}\omega$.  At each step
	there is only a single such graph.
\end{proof}

\begin{proof}[Proof of Proposition~\ref{L:convert simple words}]
	Write $G$ for the graph
	$\begin{xy}
		(0,3)*{\scriptstyle y_1}="x1",
		(0,1)*{\scriptstyle \vdots},
		(0,-3)*{\scriptstyle y_k}="x3",
		(7,0)*+R{\scriptstyle z}="y",
		"y";"x1" **\dir{-}?>*\dir{>},
		"y";(1,0) **\dir{-}?>*\dir{>},
		"y";"x3" **\dir{-}?>*\dir{>},
	\end{xy}$ and order the letters $\{y_i\}_i\cup\{z\}$ so that $z$ is minimal.
	It is enough to show  
	$\displaystyle
	G \ = \ \sum_{\sigma\in \Sigma_k}
	\begin{xy}
		(-2,-2)*{\scriptstyle y_{\sigma(1)}}="x3",
		(2,2)*{\scriptstyle \cdots}="x2",
		(6,-2)*{\scriptstyle y_{\sigma(k)}}="x1",
		(10,2)*{\scriptstyle z}="y",
		"x1";"y" **\dir{-}?<*\dir{<},
		"x2";"x1" **\dir{-}?<*\dir{<},
		"x3";"x2" **\dir{-}?<*\dir{<},
	\end{xy}$ modulo Arnold and arrow-reversing, 
	assuming $z$ and $y_i$ are all unique.
	However, this follows from the pairing calculations
	$\langle G,\,\mathcal{L}zy_{\sigma(1)}\cdots y_{\sigma(k)}\rangle = 1$
	and $\langle G,\,\mathcal{L}\omega\rangle = 0$ for all other 
	$\omega \in \mathcal{B}$.
\end{proof}

\begin{remark}
	An independent proof of Theorem~\ref{P:lower triangular} could yield an alternate
	proof that the configuration basis is a basis without the use of graph coalgebras.
	However, showing \ref{P:lower triangular} entirely in the realm of associative
	algebras, never using graphs, appears difficult.
\end{remark}

\begin{remark}
	Theorem~\ref{P:lower triangular} gives an independent proof that the Lyndon-Shirshov words
	are a basis for the Lie coalgebra $\mathrm{E}V^*$.  This yields an alternate proof that 
	the Lyndon-Shirshov words are a multiplicative basis for the shuffle algebra \cite{Radf79}.
\end{remark}

\section{Comparison with Other Lie Bases}\label{S:comparison}

In practice it is often easier to compute pairings via 
Theorem~\ref{T:bracket/cobracket}
using the bar basis $\mathcal{B}^*$ than using the graph basis 
$\mathcal{G}\mathcal{B}$.
The cobracket of a bar expression of length $(n+m)$ has only two terms of the form 
$(\text{length $n$ expression})\otimes(\text{length $m$ expression})$ -- given by 
cutting either after position $n$ or after position $m$ (and anti-commuting). 
On the other
hand, the cobracket of a graph expression could have many such terms from cutting 
various edges.
However, moving to the
bar representation of Lie coalgebras gives up the monomial dual basis 
$\mathcal{G}\mathcal{B}$.

\begin{example}\label{E:2x2y2z bars2}
	Write $[\omega]$ for the classical bracketing method constructing a Hall basis 
	from Lyndon words.  Recall from \cite{Reut93} that this is recursively defined by 
	$[\omega] = [ [\alpha], [\beta] ]$ where $\omega=\alpha\beta$ is chosen so 
	that $\alpha$ is nonempty and $\beta$ is lexicographically minimal 
	(classically it is equivalent to choose $\beta$ to so that it is the 
	longest possible such Lyndon subword).
	For example, $[xxxyyyy] = [x, [x, [ [ [ [x, y], y], y], y] ] ]$ and 
	$[xxyyxyy] = [ [x, [ [ x, y], y] ], [ [x, y], y] ]$. 
	Using this basis, Example~\ref{E:3x4y bars1} is as in Figure~3. 
	For further comparison, in Figure~4 we include also the 
	pairing matrix corresponding to Example~\ref{E:2x2y2z bars}.

	\begin{figure}[h]\label{F:3x4y bars2}
	\begin{tabular}{c || c c c c c | c}
		& $[\omega_1]$
			& $[\omega_2]$
			& $[\omega_3]$
			& $[\omega_4]$
			& $[\omega_5]$ 
			& $\ell$ \\
			\hline
			\hline
		$\omega_1^*$ &  1 &    &    &    &     &  1 \\
		$\omega_2^*$ & -4 &  1 &    &    &     & -2 \\
		$\omega_3^*$ &  6 & -3 &  1 &    &     & -1 \\
		$\omega_4^*$ & -4 &  2 & -2 &  1 &     &    \\
		$\omega_5^*$ &    &  3 & -2 &  3 &  1  &  4 \\
	\end{tabular}

	\vskip 5pt
	$$\ell = [\omega_1] + 2\,[\omega_2] - [\omega_3] - 2\,[\omega_4] + 2\,[\omega_5]$$
	\caption{Pairing with Lyndon basis for words of Example~\ref{E:3x4y}.} 
	\end{figure}

	\begin{figure}[h]\label{F:2x2y2z bars2}
	\begin{tabular}{c || c c c c c c c c c c c c c c }
		& $[\omega_1]$
			& $[\omega_2]$
			& $[\omega_3]$
			& $[\omega_4]$
			& $[\omega_5]$
			& $[\omega_6]$
			& $[\omega_7]$
			& $[\omega_8]$
			& $[\omega_9]$
			& $[\omega_{10}]$ 
			& $[\omega_{11}]$ 
			& $[\omega_{12}]$ 
			& $[\omega_{13}]$ 
			& $[\omega_{14}]$ 
			 \\
			\hline \hline
		$\omega_1^*$ &
	1& & & & & & & & & & & & &  \\
		$\omega_2^*$ &
	-2&1& & & & & & & & & & & &  \\
		$\omega_3^*$ &
	 &-1&1& & & & & & & & & & &  \\
		$\omega_4^*$ &
	 &-1& &1& & & & & & & & & & 	 \\
		$\omega_5^*$ &
	2&1&-2&-2&1& & & & & & & & & 	 \\
		$\omega_6^*$ &
	-1& &1&1&-1&1& & & & & & & & 	 \\
		$\omega_7^*$ &
	 & &-1& & & &1& & & & & & & 	 \\
		$\omega_8^*$ &
	 & &2& &-1& &-2&1& & & & & & 	 \\
		$\omega_9^*$ &
	 & &-1& &1&-2&1&-1&1& & & & & 	 \\
		$\omega_{10}^*$ &
	 & & & & &1& & &-1&1& & & & 	 \\
		$\omega_{11}^*$ &
	 & & &-1&1&-2& &-1&2&-2&1& & & 	 \\
		$\omega_{12}^*$ &
	 &2& & &-2&4& &1&-2& & &1& & 	 \\
		$\omega_{13}^*$ &
	 & & &2&-1& & &1& & &-2&-1&1&  \\
		$\omega_{14}^*$ &
	 & & &-1&1&-2& & & & &-1& &1&1 \\
	\end{tabular}
	\caption{Pairing with Lyndon basis for words of Example~\ref{E:2x2y2z}.} 
%DATA:
% Classical basis (ordered):
%  [1,[1,[2,[[2,3],3]]]] [1,[[1,[2,3]],[2,3]]] [1,[[1,[[2,3],3]],2]] [1,[[1,3],[2,[2,3]]]] [1,[[[1,3],[2,3]],2]]
%  [1,[[[[1,3],3],2],2]] [[1,2],[1,[[2,3],3]]] [[1,2],[[1,3],[2,3]]] [[1,2],[[[1,3],3],2]] [[[1,2],2],[[1,3],3]]
%  [[1,[2,[2,3]]],[1,3]] [[1,[2,3]],[[1,3],2]] [[[1,[2,3]],2],[1,3]] [[1,3],[[[1,3],2],2]]
% Lyndon basis (ordered):
%  112233 112323 112332 113223 113232 113322 121233 121323 121332 122133 122313 123132 123213 131322
% Pairing matrix:
%	 1	 0	 0	 0	 0	 0	 0	 0	 0	 0	 0	 0	0	 0	
%	-2	 1	 0	 0	 0	 0	 0	 0	 0	 0	 0	 0	0	 0	
%	 0	-1	 1	 0	 0	 0	 0	 0	 0	 0	 0	 0	0	 0	
%	 0	-1	 0	 1	 0	 0	 0	 0	 0	 0	 0	 0	0	 0	
%	 2	 1	-2	-2	 1	 0	 0	 0	 0	 0	 0	 0	0	 0	
%	-1	 0	 1	 1	-1	 1	 0	 0	 0	 0	 0	 0	0	 0	
%	 0	 0	-1	 0	 0	 0	 1	 0	 0	 0	 0	 0	0	 0	
%	 0	 0	 2	 0	-1	 0	-2	 1	 0	 0	 0	 0	0	 0	
%	 0	 0	-1	 0	 1	-2	 1	-1	 1	 0	 0	 0	0	 0	
%	 0	 0	 0	 0	 0	 1	 0	 0	-1	 1	 0	 0	0	 0	
%	 0	 0	 0	-1	 1	-2	 0	-1	 2	-2	 1	 0	0	 0	
%	 0	 2	 0	 0	-2	 4	 0	 1	-2	 0	 0	 1	0	 0	
%	 0	 0	 0	 2	-1	 0	 0	 1	 0	 0	-2	-1	1	 0	
%	 0	 0	 0	-1	 1	-2	 0	 0	 0	 0	-1	 0	1	 1
	\end{figure}

\end{example}

It immediately follows from \cite[Thm. 5.1]{Reut93} and 
Proposition~\ref{P:config as coeff} that 
the configuration pairing of $\mathcal{B}^*$ and $[\mathcal{B}]$ is 
always lower triangular with 1 on the diagonal, similar to Theorem~\ref{P:lower triangular}
for the configuration basis.
Other bases, such as the right-normed basis ${\pmb [}T_X{\pmb ]}$ of \cite{Chib06} do not
satisfy a triangularity theorem such as Theorem~\ref{P:lower triangular}.  This is
verified via explicit pairing calculations.

\begin{example}
	The basis $\llbracket \mathcal{B}\rrbracket$ of \cite{Chib06} satisfies a triangular
	pairing theorem \cite[Prop.~4.1]{Chib06}.  Since \cite{Chib06} uses Shirshov's ordering
	convention for Lyndon-Shirshov words, we reverse the ordering of the alphabet $\mathcal{A}_0$ 
	for comparison with other bases.
	Chibrikov's basis is very similar to the
	configuration basis -- when the leading letter is repeated few times, many of these 
	basis elements will differ by only a sign, for example 
	$\mathcal{L}xyzyxz = [ [z, x], [ [ [y, x], z], y] ]$ and 
	$\llbracket xyzyxz \rrbracket = [ [ [ [x, y], z], y], [x, z] ]$.
	For comparison, we give its analogous pairing matrices below.  
	(Below, we use the reverse ordering on the generators of $V$ to account for the use
	of Shirshov's definition of $\mathcal{B}$ in \cite{Chib06}.)
	
%DATA: 
% Chibrikov RN basis:
%  [3,[3,[2,[1,[2,1]]]]] [3,[2,[3,[1,[2,1]]]]] [2,[3,[3,[1,[2,1]]]]] [3,[2,[2,[1,[3,1]]]]] [2,[3,[2,[1,[3,1]]]]]
%  [2,[2,[3,[1,[3,1]]]]] [3,[1,[2,[3,[2,1]]]]] [3,[2,[1,[3,[2,1]]]]] [2,[3,[1,[3,[2,1]]]]] [3,[1,[3,[2,[2,1]]]]]
%  [1,[3,[3,[2,[2,1]]]]] [2,[1,[3,[3,[2,1]]]]] [1,[3,[2,[3,[2,1]]]]] [2,[1,[3,[2,[3,1]]]]
% Lyndon basis:
%  112233 112323 112332 113223 113232 113322 121233 121323 121332 122133 122313 123132 123213 131322
% Pairing Matrix: 
%	 1 	 0	 0	0	0	 0	 0	 0	 0	 1	 1	 0	 0	 0	
%	 0	 1	 0	0	0	 0	 1	 1	 0	 0	 0	 0	 1	 0	
%	 0	 0	 1	0	0	 0	 0	 0	 1	 0	 0	 1	 0	 0	
%	 0	 0	 0	1	0	 0	 0	 0	 0	 0	 0	 0	 0	 0	
%	 0	 0	 0	0	1	 0	 0	 0	 0	 0	 0	 0	 0	 1	
%	 0	 0	 0	0	0	 1	 0	 0	 0	 0	 0	 0	 0	 0	
%	-2	 0	 0	0	0	 0	-1	 0	 0	-2	-2	 0	-1	 0	
%	 0	-2	 0	0	0	 0	-1	-1	 0	 0	 0	 0	-1	 0	
%	 0	 0	-2	0	0	 0	 0	 0	-1	 0	 0	-1	 0	-1	
%	 0	 0	 0	0	0	 0	 1	 0	 0	 1	 1	 0	 1	 0	
%	 0	 0	 0	0	0	 0	 0	 0	 0	-1	 0	 0	 0	 0	
%	 0	 0	 0	0	0	 0	 0	 0	-1	 0	 0	 0	 0	 1	
%	 0	 0	 0	0	0	 0	-2	 0	 0	 0	 0	 0	-1	 0	
%	 0	 0	 0	0	0	-2	 0	 0	 0	 0	 0	 0	 0	-1
%
% Pairing with T_X words:
%	2	0	0	0	0	0	 1	0	0	2	2	0	 1	 0	
%	0	2	0	0	0	0	 1	1	0	0	0	0	 1	 0	
%	0	0	2	0	0	0	 0	0	1	0	0	1	 0	 1	
%	0	0	0	2	0	0	 1	1	0	1	2	0	 1	 0	
%	0	0	0	0	2	0	 0	0	1	0	0	2	 1	 1	
%	0	0	0	0	0	2	 0	0	0	0	0	0	 0	 1	
%	0	0	0	0	0	0	 2	0	0	0	0	0	 1	 0	
%	0	0	0	0	0	0	-1	1	0	0	0	0	-1	 0	
%	0	0	0	0	0	0	 0	0	1	0	0	0	 0	-1	
%	0	0	0	0	0	0	 0	0	0	1	0	0	 0	 0	
%	0	0	0	0	0	0	 0	0	0	0	1	0	 0	 0	
%	0	0	0	0	0	0	 0	0	0	0	0	1	 0	 0	
%	0	0	0	0	0	0	 0	0	0	0	0	0	 1	 0	
%	0	0	0	0	0	0	 0	0	0	0	0	0	 1	 2	
%

\begin{figure}[h]\label{F:2x2y2z bars3}
	\begin{tabular}{c || c c c c c c c c c c c c c c }
		& $\llbracket\omega_1\rrbracket$
			& $\llbracket\omega_2\rrbracket$
			& $\llbracket\omega_3\rrbracket$
			& $\llbracket\omega_4\rrbracket$
			& $\llbracket\omega_5\rrbracket$
			& $\llbracket\omega_6\rrbracket$
			& $\llbracket\omega_7\rrbracket$
			& $\llbracket\omega_8\rrbracket$
			& $\llbracket\omega_9\rrbracket$
			& $\llbracket\omega_{10}\rrbracket$ 
			& $\llbracket\omega_{11}\rrbracket$ 
			& $\llbracket\omega_{12}\rrbracket$ 
			& $\llbracket\omega_{13}\rrbracket$ 
			& $\llbracket\omega_{14}\rrbracket$ 
			 \\
			\hline \hline
$\omega_1^*$
&	 1&	  &	  &	 & &  &  & &  &  & & & & 	\\
$\omega_2^*$
&	  &	 1&	  &	 & &  &  & &  &  & & & & 	\\
$\omega_3^*$
&	  &	  &	 1&	 & &  &  & &  &  & & & & 	\\
$\omega_4^*$
&	  &	  &	  &	1& &  &  & &  &  & & & & 	\\
$\omega_5^*$
&	  &	  &	  &	 &1&  &  & &  &  & & & & 	\\
$\omega_6^*$
&	  &	  &	  &	 & & 1&  & &  &  & & & & 	\\
$\omega_7^*$
&	-2&	  &	  &	 & &  & 1& &  &  & & & & 	\\
$\omega_8^*$
&	  &	-2&	  &	 & &  &  &1&  &  & & & & 	\\
$\omega_9^*$
&	  &	  &	-2&	 & &  &  & & 1&  & & & & 	\\
$\omega_{10}^*$
&	  &	  &	  &	 & &  &-1& &  & 1& & & & 	\\
$\omega_{11}^*$
&	  &	  &	  &	 & &  &  & &  &-1&1& & & 	\\
$\omega_{12}^*$
&	  &	  &	  &	 & &  &  & &-1&  & &1& & 	\\
$\omega_{13}^*$
&	  &	  &	  &	 & &  & 2& &  &  & & &1& 	\\
$\omega_{14}^*$
&	  &	  &	  &	 & &-2&  & &  &  & & & &1	\\
\end{tabular}
\caption{Pairing with the Lyndon basis for words of Example~\ref{E:2x2y2z}.}
\end{figure}
% Chibrikov other basis:
%  [[[[1,[1,2]],2],3],3] [[[[1,[1,2]],3],2],3] [[[[1,[1,2]],3],3],2] [[[[1,[1,3]],2],2],3] [[[[1,[1,3]],2],3],2]
%  [[[[1,[1,3]],3],2],2] [[[1,2],[[1,2],3]],3] [[[[1,2],[1,3]],2],3] [[[[1,2],[1,3]],3],2] [[[[1,2],2],[1,3]],3]
%  [[[[1,2],2],3],[1,3]] [[[[1,2],3],[1,3]],2] [[[[1,2],3],2],[1,3]] [[[1,3],[[1,3],2]],2]
% Lyndon basis:
%  112233 112323 112332 113223 113232 113322 121233 121323 121332 122133 122313 123132 123213 131322
% Pairing Matrix: 
%	 1	 0	 0	0	0	 0	 0	0	 0	 0	0	0	0	0	
%	 0	 1	 0	0	0	 0	 0	0	 0	 0	0	0	0	0	
%	 0	 0	 1	0	0	 0	 0	0	 0	 0	0	0	0	0	
%	 0	 0	 0	1	0	 0	 0	0	 0	 0	0	0	0	0	
%	 0	 0	 0	0	1	 0	 0	0	 0	 0	0	0	0	0	
%	 0	 0	 0	0	0	 1	 0	0	 0	 0	0	0	0	0	
%	-2	 0	 0	0	0	 0	 1	0	 0	 0	0	0	0	0	
%	 0	-2	 0	0	0	 0	 0	1	 0	 0	0	0	0	0	
%	 0	 0	-2	0	0	 0	 0	0	 1	 0	0	0	0	0	
%	 0	 0	 0	0	0	 0	-1	0	 0	 1	0	0	0	0	
%	 0	 0	 0	0	0	 0	 0	0	 0	-1	1	0	0	0	
%	 0	 0	 0	0	0	 0	 0	0	-1	 0	0	1	0	0	
%	 0	 0	 0	0	0	 0	 2	0	 0	 0	0	0	1	0	
%	 0	 0	 0	0	0	-2	 0	0	 0	 0	0	0	0	1	

\end{example}

Note that the Lyndon basis $\mathcal{B}^*$ does not in general 
have a dual monomial basis of bracket expressions.  
This can be verified by computing the pairings of the Lyndon basis
with all bracket expressions in the vector subspace where $x$ is repeated twice and
$y$ is repeated three times.

\begin{figure}[ht]
\begin{tabular}{c || c c c c }
	 & $[[[[x,y],x],y],y]$ & $[[[[x,y],y],x],y]$ & $[[[[x,y],y],y],x]$ & $[[[x,y],y],[x,y]]$ \\
	 \hline \hline
$(xxyyy)^*$ & -1& -1& -1& \\ 
$(xyxyy)^*$ &  2&  2&  3& -1\\
\end{tabular}
\caption{Lyndon words have no dual basis of monomials.}
\end{figure}

\section{A New Shuffle Basis}\label{S:new basis}

Work similar to Theorem~\ref{T:main} and \ref{P:lower triangular} can be used to construct 
other bases of associative words for $\mathrm{E}V^*$ similar to the Lyndon-Shirshov
words.  This also yields new multiplicative bases for the shuffle algebra.

We use the ordering on words called {\it degree-lexicographic} or {\it deg-lex} by \cite{Chib06}.
In this ordering, $\omega < \upsilon$ if and only if either $\omega$ has less letters 
than $\upsilon$, or else $\omega$ and $\upsilon$ have the same number of letters and $\omega < \upsilon$
lexicographically.  For a finite alphabet, this is equivalent to using the ordering of letters 
to view words as numbers; e.g. in the alphabet $\{1<2\}$ we have $2 < 12 < 21 < 112$.

Recall (see Example~\ref{E:A_n}) that words have unique expression as a word 
of compatible simple words.  
Write $\prec$ for the ordering of $\mathcal{A}$ given by the {\it deg-lex} ordering of words in the
{\it deg-lex} ordered alphabet of simple words
(thus $(13322) \prec (13)(122) \prec (122)(13)$ unlike when working lexicographically).  It is clear that
$\prec$ is a total order on $\mathcal{A}$.

\begin{definition}
	Let $\hat{\mathcal{B}}$ be the set of finite words which have minimal $\prec$ ordering among their cyclic 
	permutations.  Define $\hat{\mathcal{B}}_m = \hat{\mathcal{B}} \cap \mathcal{A}_m$ and
	$\mathcal{L}\hat{\mathcal{B}}$, $\mathcal{G}\hat{\mathcal{B}}$ the same as in Section~\ref{S:main}.
\end{definition}

The set $\hat{\mathcal{B}}_m$ satisfies Proposition~\ref{P:Witt} and Lemma~\ref{L:order of B} 
which is enough to
make the proof of Theorem~\ref{T:main} apply.  

\begin{theorem}
	Bracket expressions $\mathcal{L}\hat{\mathcal{B}}$ and graph expressions $\mathcal{G}\hat{\mathcal{B}}$ 
	are dual vector space bases bases for $\mathrm{L}V$ and $\mathrm{E}V^*$.  

	Furthermore, Lie bracket expressions, when written in terms of $\mathcal{L}\hat{\mathcal{B}}$, have
	integer coefficients.
\end{theorem}

In the proof of 
Theorem~\ref{P:lower triangular}, applications of the Arnold identity and 
Proposition~\ref{L:convert simple words} and Corollary~\ref{C:convert simple words} reduce the
$\prec$ ordering as well as the lexicographical ordering, so an analog of 
Theorem~\ref{P:lower triangular} holds for $\hat{\mathcal{B}}$.

\begin{theorem}\label{P:hat lower triangular}
	If $\omega \prec \upsilon \in \hat{\mathcal{B}}$ then 
	$\langle\omega^*,\,\mathcal{L}\upsilon\rangle = 0$.
	Furthermore
	$\langle\omega^*,\,\mathcal{L}\omega\rangle = \pm 1$.
\end{theorem}

\begin{corollary}
	The expressions $\hat{\mathcal{B}}^*$ are a vector space basis of $\mathrm{E}V^*$. 
\end{corollary}

\begin{corollary}
	The words of $\hat{\mathcal{B}}$ are a multiplicative basis for the
	shuffle algebra.
\end{corollary}

\begin{example}
	In an alphabet with only two letters, $\mathcal{B}$ and $\hat{\mathcal{B}}$ are the same.

	In the vector subspace of brackets $x$, $y$, and $z$ each repeated twice, the change from
	$\mathcal{B}$ to $\hat{\mathcal{B}}$ affects the following words.
	$\omega_{11} = xyyzxz$ is replaced by $\hat{\omega}_{10} = xzxyyz$ and
	$\omega_{13} = xyzyxz$ is replaced by $\hat{\omega}_{11} = xzxyzy$ and 
	$\omega_{10}\cdots\omega_{14}$ are reordered.

	The corresponding portion of the pairing matrix for 
	$\hat{\mathcal{B}}^*$ and $\mathcal{L}\hat{\mathcal{B}}$ has the same 
	number of non-zero off-diagonal entries as that of $\mathcal{B}^*$ and 
	$\llbracket\mathcal{B}\rrbracket$.

	\begin{figure}[h]\label{F:2x2y2z bars4}
	\begin{tabular}{c || c c c c c c c c c c c c c c }
		& $\mathcal{L}\hat{\omega}_1$
		& $\mathcal{L}\hat{\omega}_2$
		& $\mathcal{L}\hat{\omega}_3$
		& $\mathcal{L}\hat{\omega}_4$
		& $\mathcal{L}\hat{\omega}_5$
		& $\mathcal{L}\hat{\omega}_6$
		& $\mathcal{L}\hat{\omega}_7$
		& $\mathcal{L}\hat{\omega}_8$
		& $\mathcal{L}\hat{\omega}_9$
		& $\mathcal{L}\hat{\omega}_{10}$ 
		& $\mathcal{L}\hat{\omega}_{11}$ 
		& $\mathcal{L}\hat{\omega}_{12}$ 
		& $\mathcal{L}\hat{\omega}_{13}$ 
		& $\mathcal{L}\hat{\omega}_{14}$ 
			 \\
			\hline \hline
	$\hat{\omega}_1^*$ &
 	 1& & & & & & & & & & & & & 	 \\
	$\hat{\omega}_2^*$ &
	 &1& & & & & & & & & & & & 	 \\
	 $\hat{\omega}_3^*$ &
	 & &1& & & & & & & & & & & 	 \\
	 $\hat{\omega}_4^*$ &
	 & & &1& & & & & & & & & & 	 \\
	 $\hat{\omega}_5^*$ &
	 & & & &1& & & & & & & & & 	 \\
	 $\hat{\omega}_6^*$ &
	 & & & & &1& & & & & & & & 	 \\
	 $\hat{\omega}_7^*$ &
	-2& & & & & &-1& & & & & & & 	 \\
	 $\hat{\omega}_8^*$ &
	 &-2& & & & & &-1& & & & & & 	 \\
	 $\hat{\omega}_9^*$ &
	 & &-2& & & & & &-1& & & & & 	 \\
	 $\hat{\omega}_{10}^*$ &
	 & & &-2& & & & & &-1& & & & 	 \\
	 $\hat{\omega}_{11}^*$ &
	 & & & &-2& & & & & &-1& & & 	 \\
	 $\hat{\omega}_{12}^*$ &
	 & & & & &-2& & & & & &-1& & 	 \\
	 $\hat{\omega}_{13}^*$ &
	 & & & & & &1&1&1& & & &-1& 	 \\
	 $\hat{\omega}_{14}^*$ &
	 & & & & & & &1&2& & & & &-1	 \\
	\end{tabular}
	\caption{Pairing analogous to Example~\ref{E:2x2y2z}.} 
\end{figure}

%DATA:  
% words: 112233 112323 112332 113223 113232 113322 121233 121323 121332 131223 131232 122133 123132
% brackets:
%  [[[[[2,1],1],2],3],3] [[[[[2,1],1],3],2],3] [[[[[2,1],1],3],3],2] [[[[[3,1],1],2],2],3] [[[[[3,1],1],2],3],2]
%  [[[[[3,1],1],3],2],2] [[[[2,1],3],3],[2,1]] [[[[3,1],2],3],[2,1]] [[[[3,1],3],2],[2,1]] [[[[2,1],2],3],[3,1]]
%  [[[[2,1],3],2],[3,1]] [[[[3,1],2],2],[3,1]] [[[3,1],3],[[2,1],2]] [[[3,1],2],[[2,1],3]]
% pairing matrix:
%	 1	 0	 0	 0	 0	 0	 0	 0	 0	 0	 0	 0	 0	 0
%	 0	 1	 0	 0	 0	 0	 0	 0	 0	 0	 0	 0	 0	 0
%	 0	 0	 1	 0	 0	 0	 0	 0	 0	 0	 0	 0	 0	 0
%	 0	 0	 0	 1	 0	 0	 0	 0	 0	 0	 0	 0	 0	 0
%	 0	 0	 0	 0	 1	 0	 0	 0	 0	 0	 0	 0	 0	 0
%	 0	 0	 0	 0	 0	 1	 0	 0	 0	 0	 0	 0	 0	 0
%	-2	 0	 0	 0	 0	 0	-1	 0	 0	 0	 0	 0	 0	 0
%	 0	-2	 0	 0	 0	 0	 0	-1	 0	 0	 0	 0	 0	 0
%	 0	 0	-2	 0	 0	 0	 0	 0	-1	 0	 0	 0	 0	 0
%	 0	 0	 0	-2	 0	 0	 0	 0	 0	-1	 0	 0	 0	 0
%	 0	 0	 0	 0	-2	 0	 0	 0	 0	 0	-1	 0	 0	 0
%	 0	 0	 0	 0	 0	-2	 0	 0	 0	 0	 0	-1	 0	 0
%	 0	 0	 0	 0	 0	 0	 1	 1	 1	 0	 0	 0	-1	 0		
%	 0	 0	 0	 0	 0	 0	 0	 1	 2	 0	 0	 0	 0	-1		

\end{example}

It is not possible to extend Chibrikov's definition of $\llbracket \omega \rrbracket$ to $\hat{\mathcal{B}}$; 
however the classical bracketing $[\omega]$ does extend to $\hat{\mathcal{B}}$.  In fact,
even the proof of 
the triangularity theorem \cite[Thm. 4.9]{Reut93} appears to extend to $[\hat{\mathcal{B}}]$.  
The result of pairing $\hat{\mathcal{B}}^*$ and $[\hat{\mathcal{B}}]$ 
is similar to Figure~4.

The basis $\mathcal{L}\hat{\mathcal{B}}$ has a good property with respect to vector space quotient maps.
Let $\phi:V\to W$ be a vector space quotient map, and suppose that $V$ and $W$ have ordered 
bases $\{v_1,\dots,v_n\}$ and $\{w_1,\dots,w_m\}$ compatible with $\phi$ so that 
$\phi:\{v_1,\dots,v_n\} \to \{w_1,\dots,w_m\}$ with $\phi(v_i) \le \phi(v_j)$ for $i<j$.  
Suppose further that $\phi^{-1}(w_1) = \{v_1\}$. 
Write $\hat{\mathcal{B}}_V$, $\hat{\mathcal{B}}_W$ for the sets of words $\hat{\mathcal{B}}$ 
with respect to the ordered alphabets $\{v_i\}$ and $\{w_j\}$.
Define the map $\hat{\phi}:\hat{\mathcal{B}}_V \to \hat{\mathcal{B}}_W\cup\{0\}$ by 
$\hat{\phi}(a_1\cdots a_k) = \phi(a_1)\cdots\phi(a_k)$ if $\phi(a_1)\cdots\phi(a_k)\in\hat{\mathcal{B}}_W$
and $0$ otherwise. 

\begin{proposition}
	The map $\mathrm{L}\phi:\mathrm{L}V\to \mathrm{L}W$ is given on the basis
	$\mathcal{L}\hat{\mathcal{B}}_V$ by 
	$\mathcal{L}\hat{\phi}:\mathcal{L}\hat{\mathcal{B}}_V \to \mathcal{L}\hat{\mathcal{B}}_W\cup\{0\}$ 
\end{proposition}

The above proposition can be used to quickly calculate $\mathcal{L}\hat{\mathcal{B}}$.  For
example, $[[[v_2,v_1],v_2],[v_2,v_1]] \in \mathcal{L}\hat{\mathcal{B}}$ so 
$[[[v_i, v_1], v_j], [v_k, v_1]] \in \mathcal{L}\hat{\mathcal{B}}$ as well for all $i,j,k\neq 1$. 
Neither $[\mathcal{B}]$, $\llbracket\mathcal{B}\rrbracket$, nor $[\hat{\mathcal{B}}]$ have 
this property.

\section{Future Directions}\label{S:future}

\subsection{Pairing matrix formula}
The off-diagonal elements of the pairing matrix of $\mathcal{B}^*$ and
the configuration basis $\mathcal{L}\mathcal{B}$ are due to applications of the Arnold
identity in the proof of Theorem~\ref{P:lower triangular}.  A more careful analysis, 
keeping track of signs and counting occurences should lead to a explicit formulas
writing $\mathcal{B}^*$ in terms of $\mathcal{G}\mathcal{B}$, which 
could be used to write the pairing matrix without any pairing computations.  This would
yield marked computational improvements, since there are no known formulas giving
the analogous pairing matrix for either $[\mathcal{B}]$ or $\llbracket\mathcal{B}\rrbracket$.
Since the applications of Arnold in Theorem~\ref{P:lower triangular} run from left
to right, this computation will be simplest for $\hat{\mathcal{B}}$ which gathers small
simple words on the left side.

\subsection{Gr\"obner basis implimentation}
The triangularity theorems (\ref{P:lower triangular} and \ref{P:hat lower triangular}) 
imply that $\mathcal{L}\mathcal{B}$ and
$\mathcal{L}\hat{\mathcal{B}}$ may be used in Gr\"obner basis calculations.  This should be
implimented in a computer algebra software platform such as {\it Sage} or {\it GAP}.

\subsection{Dual monomial basis}
It is unclear whether there is no dual monomial basis of 
Lie bracket expressions of $\mathrm{L}V$ and associative (bar) words of $\mathrm{E}V^*$.
As noted in Figure 6, the Lyndon-Shirshov words do not have a dual monomial basis.  Neither will
purely left-normed bracket expressions (e.g. those of the form $[[[\cdot,\cdot],\cdots],\cdot]$
-- expressions with 3 of one generator and 4 of another provides a counter-example).
However, ad-hoc dual monomial bases can be found for examples of computable size.

\appendix

\section{Bases}

Below we list the bases used in Figures~2, 4, and 5
presented previously.  Recall that in our computation of $\llbracket\mathcal{B}\rrbracket$,
we reversed the order of the basis elements to account for the use of Shirshov's ordering 
convention by \cite{Chib06}.

\begin{tabular}{c c c c }
	& $\mathcal{L}\omega$ & $[\omega]$ & $\llbracket\omega\rrbracket$ \\
$\omega_1 = xxyyzz$ & $[[[[[y,x],x],y],z],z]$ & $[x,[x,[y,[[y,z],z]]]]$ & $[[[[x,[x,y]],y],z],z]$ \\
$\omega_2 = xxyzyz$ & $[[[[[y,x],x],z],y],z]$ & $[x,[[x,[y,z]],[y,z]]]$ & $[[[[x,[x,y]],z],y],z]$ \\
$\omega_3 = xxyzzy$ & $[[[[[y,x],x],z],z],y]$ & $[x,[[x,[[y,z],z]],y]]$ & $[[[[x,[x,y]],z],z],y]$ \\
$\omega_4 = xxzyyz$ & $[[[[[z,x],x],y],y],z]$ & $[x,[[x,z],[y,[y,z]]]]$ & $[[[[x,[x,z]],y],y],z]$ \\
$\omega_5 = xxzyzy$ & $[[[[[z,x],x],y],z],y]$ & $[x,[[[x,z],[y,z]],y]]$ & $[[[[x,[x,z]],y],z],y]$ \\
$\omega_6 = xxzzyy$ & $[[[[[z,x],x],z],y],y]$ & $[x,[[[[x,z],z],y],y]]$ & $[[[[x,[x,z]],z],y],y]$ \\
$\omega_7 = xyxyzz$ & $[[[[y,x],z],z],[y,x]]$ & $[[x,y],[x,[[y,z],z]]]$ & $[[[x,y],[[x,y],z]],z]$ \\
$\omega_8 = xyxzyz$ & $[[[[z,x],y],z],[y,x]]$ & $[[x,y],[[x,z],[y,z]]]$ & $[[[[x,y],[x,z]],y],z]$ \\
$\omega_9 = xyxzzy$ & $[[[[z,x],z],y],[y,x]]$ & $[[x,y],[[[x,z],z],y]]$ & $[[[[x,y],[x,z]],z],y]$ \\
$\omega_{10} = xyyxzz$ & $[[[z,x],z],[[y,x],y]]$ & $[[[x,y],y],[[x,z],z]]$ & $[[[[x,y],y],[x,z]],z]$ \\
$\omega_{11} = xyyzxz$ & $[[z,x],[[[y,x],y],z]]$ & $[[x,[y,[y,z]]],[x,z]]$ & $[[[[x,y],y],z],[x,z]]$ \\
$\omega_{12} = xyzxzy$ & $[[[z,x],y],[[y,x],z]]$ & $[[x,[y,z]],[[x,z],y]]$ & $[[[[x,y],z],[x,z]],y]$ \\
$\omega_{13} = xyzyxz$ & $[[z,x],[[[y,x],z],y]]$ & $[[[x,[y,z]],y],[x,z]]$ & $[[[[x,y],z],y],[x,z]]$ \\
$\omega_{14} = xzxzyy$ & $[[[[z,x],y],y],[z,x]]$ & $[[x,z],[[[x,z],y],y]]$ & $[[[x,z],[[x,z],y]],y]$ \\
\end{tabular}

%\bibliographystyle{habbrv}
%\bibliography{references}
%\end{document}

\end{document}